\documentclass[12pt]{article}
\usepackage{amsmath, amsfonts, amsthm, amssymb, color, hyperref, extarrows, enumitem}
 
 \newtheorem{theorem}{Theorem}[section]

\newtheorem{cor}[theorem]{Corollary}

\usepackage[T1]{fontenc}

 \theoremstyle{definition}
 
\usepackage{extarrows}

\newtheorem{lemma}[theorem]{Lemma}

\theoremstyle{remark}
 
\numberwithin{equation}{section}

\usepackage{extarrows}

\DeclareMathOperator{\Hi}{O}
\DeclareMathOperator{\Imagine}{Im}
\DeclareMathOperator{\Real}{Re}

\usepackage{cite}
 
\hypersetup{
	colorlinks   = true,
	citecolor    = magenta}
	
\begin{document}

\title{\vspace{-1.2cm} \bf Boundary asymptotics of the relative Bergman kernel metric for curves} 
\author{Robert Xin Dong}
\date{}
\maketitle

\begin{abstract}
We study the behaviors of the relative Bergman kernel metrics on holomorphic families of degenerating hyperelliptic Riemann surfaces and their Jacobian varieties. Near a node or cusp, we obtain precise asymptotic formulas with explicit coefficients. In general the Bergman kernels on a given cuspidal family do not always converge to that on the regular part of the limiting surface, which is different from the nodal case. It turns out that information on both the singularity and the complex structure contributes to various asymptotic behaviors of the Bergman kernel. Our method involves the classical Taylor expansion for Abelian differentials and period matrices. 
 \end{abstract}
 
\renewcommand{\thefootnote}{\fnsymbol{footnote}}
\footnotetext{\hspace*{-7mm} 
\begin{tabular}{@{}r@{}p{16.5cm}@{}}
& Keywords. Bergman kernel metric, cusp, degeneration, Hodge theory, hyperelliptic Riemann surface, Jacobian variety, node, Ohsawa-Takegoshi $L^2$ extension, period matrix, positivity of direct image\\
& Mathematics Subject Classification. Primary 32A25; Secondary 32G20, 32G15, 14H45, 32Q28, 14D06

\end{tabular}}

\section*{Introduction}

On an $n$-dimensional complex manifold $X$, the Bergman kernel is a reproducing kernel of the space of square integrable holomorphic top forms. For the canonical bundle $\mathcal K$, which is the $n$th exterior power of the cotangent bundle on $X$, the Bergman kernel is defined as  
 \begin{equation} \label{def} 
\kappa:=\sum e_j \otimes \overline {e_j},
 \end{equation} 
 where $\{e_j\}$ is a complete orthonormal basis of $H_{L^2}^0(X, \mathcal K)$. Thus, the Bergman kernel is a canonical volume-form determined only by the complex structure.  
 As the complex structure deforms, the log-plurisubharmonic variation results of the Bergman kernels on pseudoconvex domains were obtained by Maitani and Yamaguchi \cite{MY} and later by Berndtsson \cite{B1, B2}.
More generally, further important developments on general Stein manifolds and complex projective algebraic manifolds \cite{B2, Ts, BP, LY, PT, Ta, DWZZ} state certain positivity properties of the direct images of the relative canonical bundles, and turn out to have intimate connections with the space of K\"{a}hler metrics \cite{B3, DS, DS2, T}, the invariance of plurigenera \cite{Le, Si98, Si02, P07, BP2}, and the sharp Ohsawa-Takegoshi theorem \cite{OT, O95, Si, Bl13, GZ1, BL, Ca, GZ2, GZ3, Dem16}. Conversely, the Ohsawa-Takegoshi theorem with optimal constant was applied in \cite{GZ1, Ca} to prove the variation results of the Bergman kernels. 
 
\medskip

Consider the family 
$
\mathfrak{X} \xlongrightarrow {\pi} 
\Delta,
$
where $\mathfrak{X}$ is a complex $(n+1)$-dimensional Stein manifold fibred over the unit disc $\Delta\subset\mathbb{C}$, and $\pi$ is a surjective holomorphic map with connected fibers. Moreover, $\pi$ is a holomorphic submersion over $\Delta^*:=\Delta\setminus \{0\}$. For $\lambda\in \Delta^*$, the fibres $X_\lambda=\pi^{-1}(\lambda)$ are $n$-dimensional Stein manifolds, whose Bergman kernels on the diagonal are denoted by $\kappa_{X_\lambda}$. Let $(z^1, \ldots, z^n, \lambda)$ be a coordinate of $X_\lambda \times \Delta^* $, which is a local trivialization of the fibration, and write $\kappa_{X_\lambda}:=k_{\lambda}(z) (dz^1\wedge \cdots \wedge dz^n) \otimes  
(d\overline{z^1}\wedge \cdots \wedge d\overline{z^n})  $ locally. Then, the above mentioned positivity results imply that $\psi:=\log k_{\lambda}(z)$ is subharmonic
with respect to $ \lambda \in \Delta^*$, i.e.,
\begin{equation} \label{SH}
  \frac{\partial ^2 \psi }{\partial {\lambda} \partial {\bar \lambda} } \,  \geq 0,
\end{equation} 
if the Bergman kernel is not identically zero. Here as $\lambda$ varies, $e^\psi$ induces on $\mathcal K_{\mathfrak{X} /\Delta}$ the so-called relative Bergman kernel metric, which is represented by different local functions, given different local trivializations (see \cite{BP}).
When the fibration $\pi$ has a singular fiber $X_0$, our goal is to quantitatively characterize at degeneration the asymptotic of $\psi$, as $\lambda \to 0$.
  
\medskip

In fact, the study of the asymptotic analysis of period integrals -- and, in particular, the Abelian differentials on algebraic curves --  is a very classical subject and has been pursued by many experts in the field of Hodge theory \cite{K, KK, Zu, De,  V}, especially the nilpotent orbit theorem \cite{GS, S} in the variations of Hodge structures. In \cite{Gr}, the (pluri)subharmonicity in the base direction of the above $\psi$ was described as ``a pleasant surprise'' by Griffiths. It is known classically that the earlier works of Griffiths \cite{Gr0, Gr}, Fujita \cite{Fu}, as well as other important results in algebraic geometry including \cite{Kaw, Ko} played a decisive role in understanding the behavior of $\psi$. For more background on the $L^2$-analysis in several complex variables, and its interaction with Hodge theory, see \cite{B5, Dem12, O17, O18, P19, O20, O20!}.   

\medskip

In the theory of Riemann surfaces and their moduli space, the degenerations of 
analytic differentials have also been extensively studied via a general method called the pinching coordinate (see \cite{F, Y, Ma} and the references therein). 
The spaces of degenerating Riemannn surfaces correspond to paths in the moduli space leading to the boundary points, and are obtained from compact surfaces by shrinking finitely many closed loops to points, called nodes. Near nodal singularities of general curves, Wentworth \cite{We} obtained precise estimates for the Arakelov metric with applications to arithmetic geometry and string theory (see also \cite{dJ, J}); Habermann and Jost \cite{HJ, HJ2} studied the behavior of the Bergman kernels and their induced $L^2$ metrics on Teichm\"{u}ller spaces, with a strong motivation in minimal surface theory. Specifically, the result in \cite{HJ} shows that the Bergman kernels on a degenerating family of compact Riemann surfaces converge to that on the normalization of the limiting nodal curve, with the second term having asymptotic growth $(-\log|\lambda|)^{-1}$, as $\lambda \to 0$.

\medskip

On the other hand, there exist worse singularities such as cusps, with which the pinching coordinate method cannot deal. Boucksom and Jonsson \cite{BJ17} studied the asymptotics of volume forms on degenerating compact manifolds and established a measure-theoretic version of the Kontsevich-Soibelman conjecture, which deals with the limiting behavior of a family of Calabi-Yau manifolds approaching a ``cusp'' in the moduli space boundary. In \cite{D1, D2, D3, D4}, the author obtained quantitative results for the Bergman kernels on Legendre and other degenerating families of elliptic curves, by using elliptic functions as well as a method based on the Taylor series expansion of Abelian differentials and period matrices (see \cite{CMSP}).

 \medskip

In this paper, for the fibers $X_\lambda$ being algebraic curves or their Jacobian varieties, we determine the precise asymptotic behaviors of $\psi$ at degeneration, using the classical Taylor expansion method.
Our first motivation is to write down explicitly the asymptotic coefficients, which involve the complex structure information and reflect the geometry of the base varieties and their singularities, 
for various families of hyperelliptic curves degenerating to singular ones with nodes or cusps.
Our second motivation is to investigate whether similar convergence results hold true for the Bergman kernel on cuspidal degenerating families of curves, in comparison with \cite{HJ, HJ2, We, BJ17, Shi}. It turns out that this is not always the case as seen in Theorem \ref{potential-g-cuspI}.
However, in Theorem \ref{potential-g-cuspII} we find that the Bergman kernels on some cuspidal families of curves indeed converge at degeneration. As a last motivation, we apply the results on higher genus curves to the study of their Jacobians, and generalize the results in \cite{D1, D2, D4} toward higher dimensions.

 \medskip

The organization of this paper is as follows. 
In Section 1, we state our main results on the degeneration of the Bergman kernel. In Section 2, we study the Bergman kernel on the normalization of general algebraic curves. In Sections 3--4, we work on nodal curves. In Sections 5--6 and 7--8, we work on cuspidal curves of types I and II, respectively. Our results on the Jacobian varieties are proved in Section 9.

\section{Main results} 

Throughout this paper, a smooth algebraic curve means a Riemann surface, and we consider the complex analytic families of hyperelliptic curves $X_{\lambda}:=\{(x,y)\in \mathbb C^2\, |\,y^2=h_{\lambda}(x)P(x)\}$ of genus $g\geq 2$,  parameterized by $\lambda \in \Delta^*$.
Here, $P$ denotes a polynomial of degree $2g- 2$ with complex roots $a_j$ such that $1<|a_1|<|a_2|<\cdots<|a_{2g- 2}|$. 
For each $\lambda \in \Delta^*$, $h_{\lambda}$ is a degree $3$ polynomial with distinct roots of absolute values less than or equal to $1$. It is well known that on the smooth curve $X_{\lambda}$ there exist  globally defined Abelian differentials
\begin{equation} \label{basis}
\omega_i:=\frac{x^{i-1}dx}{y}, \, i=1,\ldots,g.
\end{equation}
 By its classical construction, the hyperelliptic curve $X_{\lambda}$ can be realized as a 2-sheeted ramified covering of the Riemann sphere $\mathbb P^1$ (see \cite{CMSP}). Take a homology basis $\delta_i$ and $\gamma_j$ of $H_1(X_{\lambda}, \mathbb Z)$ such that their intersection numbers are $\delta_i\cdot \delta_j=0= \gamma_i\cdot \gamma_j$, and $\delta_i\cdot\gamma_j=\delta_{i,j}=-\gamma_j \cdot \delta_i$ (here $\delta_{i,j}$ is the Kronecker $\delta$). 
Then, for the above $\omega_i \in H^0(X_{\lambda}, \mathcal K)$, we set
\begin{equation} \label{homology basis}
A_{i, j}:=\int_{\delta_j}\omega_i, \quad B_{i, j}:=\int_{\delta_j}\gamma_i.
\end{equation}
By the Hodge-Riemann bilinear relation,  the matrix $(A_{i j})_{1\leq i,j\leq g}$ is invertible and define the normalized period matrix $Z:=A^{-1}B$. Then, $Z$ is symmetric and has a positive definite imaginary part, i.e., $\Imagine Z >0.$ With respect to the chosen homology basis, we call the matrices $A, B$ and $Z$, the $A$-period matrix, $B$-period matrix, and normalized period matrix, respectively.
 Moreover,  we use the symbol ``$\sim$'' to denote that the ratio of its both sides tends to  $1$, as  $\lambda \to 0$; in the case of matrices,   it refers to the entrywise ratio.

\paragraph{Nodal cases} 
In affine coordinates, consider a family of hyperelliptic curves 
\begin{equation} \label{X-node}
X_{\lambda}:=\{(x,y)\in \mathbb C^2\, |\,y^2=x(x-{\lambda})(x-1) P(x)\}.
\end{equation} 
As $\lambda \to 0$, $X_{\lambda}$ degenerates to a singular curve $X_0$ with a non-separating node. 
 The normalization of $X_0$ is a smooth curve $Y:=\{(x,y)\in \mathbb C^2\, |\,y^2=(x-1)P(x)\}$, whose $A$-period matrix and normalized period matrix are denoted by $A_0$ and $Z_0$, respectively. By the $L^2$ removable singularity theorem, $Y$ and the regular part of $X_0$ have the same Bergman kernel, denoted by $\kappa_{0}$.
We write $\kappa_{X_\lambda}=k_{\lambda}(z)dz\otimes  
d\bar{z}$ and $\kappa_{0}=k_0(z)dz\otimes  
d\bar{z}$ in the local coordinate $z:=\sqrt{x}$ near $(0, 0)$; see their precise formulas in \eqref{Bergman-g-node}  and  \eqref{Bergman-reg-node}.

\begin{theorem} [Nodal cases] \label{potential-g-node}  
For $X_{\lambda}$ defined by \eqref{X-node}, as $\lambda\rightarrow 0$, it holds that
\begin{itemize}
\item [$(i)$] $\kappa_{X_\lambda}  \to \kappa_{0}$;

\item [$(ii)$] for $0\neq |z|$ small, 
$$
\psi - \log  {k}_0(z)   \sim     \frac{\pi}{-\log |\lambda|}
\frac{ { 1- 2 \Real  \sum _{i=1}^{g-1}   \left((\Imagine  Z_0 )^{-1}   \Imagine (A_0^{-1} \star)  \right)_{i} {z}^{2i}}{}}{  \sum _{i,j=1}^{g-1}  ((\Imagine  Z_0 )^{-1})_{i, j}   ( z^{i} \overline{z}^{j})^2},
$$
where $\star$ is a column vector with $g-1$ rows whose entries all equal $ {-2i\sqrt{P(0)}^{-1}}{}$.

\end{itemize}
  \end{theorem}

Part $(i)$ is essentially due to Habermann and Jost \cite{HJ}, who in fact have used the pinching coordinate method to study general (possibly non-hyperelliptic) algebraic curves degenerating to singular ones with separating or non-separating nodes.
To write down the explicit asymptotic coefficients in Part $(ii)$, we rely on the study of genus two curves in our Theorem \ref{potential-2-node}.

\paragraph{Cuspidal cases} Since in \cite{F, Y, Ma, We, HJ, HJ2, J} only the nodal degeneration was considered, we continue investigating whether similar convergence results hold true for the Bergman kernel on cuspidal degenerating families of curves, and find that in general this is not the case. 
For the singular curve $X_{0}:=\{(x,y)\in \mathbb C^2\, |\,y^2=x^3 P(x)\}$ with an ordinary cusp at $(0, 0)$, its normalization is the smooth curve $Y:=\{(x,y)\in \mathbb C^2\, |\,y^2=xP(x)\}$, whose $A$-period matrix and normalized period matrix are denoted by $A_0$ and $Z_0$, respectively. Then, $Y$ and the regular part of $X_0$ have the same Bergman kernel, denoted by $\kappa_{0}$. In the following two cases, we explore different  families of hyperelliptic curves $X_{\lambda}$ that give rise to $X_{0}$, as $\lambda\rightarrow 0$.
Similarly, we write $\kappa_{X_\lambda}=k_{\lambda}(z)dz\otimes  
d\bar{z}$ and $ \kappa_{0}=k_0(z)dz\otimes  
d\bar{z}$ in the local coordinate $z:=\sqrt{x}$ near $(0, 0)$; see their precise formulas in Section 2.

 \begin{theorem} [Cuspidal cases, I] \label{potential-g-cuspI} 
For the hyperelliptic curves 
\begin{equation} \label{X-cuspI}
X_{\lambda}:=\{(x,y)\in \mathbb C^2\, |\,y^2=x(x^2-{\lambda})  P(x)\},
\end{equation} 
 as $\lambda\rightarrow 0$, for $0\neq |z|$ small, it holds that
\begin{itemize}
\item [$(i)$]  $
   \quad {k}_{\lambda}(z) \to   {k}_0(z) + 4 |  z^{ 4} P(z^2)|^{-1} , \quad i.e., \,
\kappa_{X_\lambda}  \not \to \kappa_0;
$
\end{itemize}
 $$
  \noindent{}{ {(ii)}    } \hspace{1.5cm} \psi - \log  \left(  {k}_0(z) +\frac{4}{|  z^4 P(z^2)|} \right)   \sim   \frac{- 2 \Real  \sum_{i=1}^{g-1}   \left((\Imagine  Z_0 )^{-1}   \Imagine (A_0^{-1} (\diamondsuit-\sqrt{-1} \diamond))  \right)_{i} {z}^{2i} }{1  +\sum _{i,j=1}^{g-1}  ((\Imagine  Z_0 )^{-1})_{i, j}   ( z^{i} \overline{z}^{j})^2},
 $$
where $\diamond$ and $\diamondsuit$ are column vectors\footnote{See their definitions in Lemmata \ref{period-A-hyp-cuspI}  and \ref{period-B-hyp-cuspI}} and with $g-1$ rows involving $\lambda$.

  \end{theorem}

\begin{theorem} [Cuspidal cases, II]  \label{potential-g-cuspII}   
For the hyperelliptic curves 
\begin{equation} \label{X-cuspII}
X_{\lambda}:=\{(x,y)\in \mathbb C^2\, |\,y^2=x(x-{\lambda})(x-{\lambda^2}) P(x)\},
\end{equation} 
as $\lambda\rightarrow 0$, it holds that
\begin{itemize}

\item [$(i)$] $\kappa_{X_\lambda}  \to \kappa_{0}$;

\item [$(ii)$] for $0\neq |z|$ small, 
 $$
\psi - \log {k}_0(z) \sim   \frac{4\pi}{{k}_0(z) |z^4 P(z^2)|}  \cdot  \frac{1}{-\log |\lambda|}.
 $$

  \end{itemize}
  \end{theorem}

 \noindent{}{\bf Remarks.}

\begin{enumerate}[label=(\roman*)]
 
\item [(a)]

In Theorem \ref{potential-g-cuspI}, $(ii)$, the right hand side in the expansion of $\psi $ is harmonic in $\lambda$ of order $\Hi(\lambda^{1/4})$, which is different from Theorems \ref{potential-g-node} and \ref{potential-g-cuspII} where the growth order $(-\log|\lambda|)^{-1}$ appears. The harmonicity in the Bergman kernel variation detects the triviality of   holomorphic fibrations, with applications to the Suita conjecture (see \cite{B4, BL, D6, DT,  DTZ, Su}).

 \item  [(b)]
The proofs of Theorems \ref{potential-g-cuspI} and \ref{potential-g-cuspII} rely on the proofs of Theorems \ref{potential-2-cuspI} and \ref{potential-2-cuspII}, respectively, for genus-two curves with more explicit asymptotic coefficients when $P=(x-a)(x-b)$. In particular, Theorem \ref{potential-2-cuspII} says that the right hand side in Theorem \ref{potential-g-cuspII}, $(ii)$, reduces to 
$$
\frac{\pi \cdot \Imagine  \tau    }{-\log|\lambda|\cdot |z|^4},
 $$
 where $\tau$ is the period (scalar) of the elliptic curve $\{(x,y)\in \mathbb C^2\, |\,y^2= x (x-a)(x-b)\}.$

  \item [(c)]  In Theorems \ref{potential-g-node} and \ref{potential-g-cuspII}, the coefficients of the first and second terms in the expansions of $\psi$ depend only on the information away from the singularity.

\end{enumerate}
 
 The boundary behavior of canonically defined plurisubharmonic functions often contains significant information, in general.
 Recall that (see \cite{Lel})   the Lelong number for a plurisubharmonic function $u$ in a neighbourhood of $0 \in \mathbb C^n$ is defined  as 
  $$
  \lim_{r \to 0^+} \frac{\sup_{|z|=r} u(z)} { \log r},
  $$
  which is  the supreme of all numbers $\nu \geq  0$ such that
$u(z) \leq \nu  \log |z| + \Hi (1)$ near $0$.
For a holomorphic family of Riemann surfaces $X_{\lambda}$ 
 parameterized by $\lambda \in \Delta^*$, each fiberwise Bergman kernel is locally written as $\kappa_{X_\lambda}:=k_{\lambda}(z)dz\otimes  
d\bar{z}$, in some coordinate $z$ for some function $k_{\lambda}$. By \eqref{SH}, the potential  $\psi:=\log k_{\lambda}(z)$ is subharmonic
with respect to $ \lambda \in \Delta^*$. Moreover, the boundedness of $\psi $ near $\lambda =0$ proved in Theorems \ref{potential-g-node}, \ref{potential-g-cuspI} and \ref{potential-g-cuspII} yields the following corollary.

  \begin{cor} \label{Lelong}  
 For each family of hyperelliptic curves defined by \eqref{X-node}, \eqref{X-cuspI}  and \eqref{X-cuspII}, as $\lambda\rightarrow 0$, $\psi:=\log k_{\lambda}(z)$ extends to a subharmonic function in $\lambda \in \Delta$, where the local coordinate is $z:=\sqrt{x}\neq 0$. Moreover, $\psi$ has Lelong number zero at $\lambda =0$.

  \end{cor}

 Thus, in our studies of curves, the subharmonic function $\psi$ tends to be small, not big, near the singular fibers (cf. \cite{BP}).
  
\paragraph{Jacobian varieties} 
 For a curve $X_{\lambda}$ of genus $g\geq 2$, its Jacobian variety $\text{Jac}(X_{\lambda})$ is a $g$-dimensional complex torus whose Bergman kernel by the definition \eqref{def} is written as $ \mu_{\lambda}  (dw^1\wedge \cdots \wedge dw^g) \otimes  
(d\overline{w^1}\wedge \cdots \wedge d\overline{w^g}) $, under the coordinate $(w^1, \ldots, w^g)$ induced from $\mathbb C^g$. For each family of hyperelliptic curves defined by \eqref{X-node}, \eqref{X-cuspI}, or \eqref{X-cuspII}, we determine the asymptotic behaviors of the Bergman kernels on the naturally associated family of Jacobians $\text{Jac}(X_{\lambda})$ as follows. 
 
 \begin{theorem}  [Jacobian varieties] \label{Jacobian}   
 As $\lambda\rightarrow 0$, it holds that
 $$
 \log \mu_{\lambda}  =   
\begin{cases}
-\log   ( {-\log |\lambda|}) + \log \frac{\pi}{\det (\Imagine  Z_0 )}     + \Hi \left(  (\log|\lambda|)^{-1}  \right), &  \text{for } X_{\lambda} \text{ in }\eqref{X-node}\, \text{or}\, \eqref{X-cuspII} ;\\

\medskip

-\log \det (\Imagine  Z_0 ) + \Hi ( \lambda^{ {1}/{2}}),   &  \text{for } X_{\lambda} \text{ in }\eqref{X-cuspI}.

\end{cases}
 $$
 
  \end{theorem}

 It is worth mentioning that as the curves $X_{\lambda}$ in \eqref{X-cuspI} degenerate to a cuspidal one, the family of their Jacobian varieties $\text{Jac}(X_{\lambda})$ 
 has a smooth fiber at   $\lambda = 0$. However, for two other families of curves $X_{\lambda}$ in \eqref{X-node} and \eqref{X-cuspII}, their Jacobian varieties $\text{Jac}(X_{\lambda})$ indeed degenerate, as   $\lambda \to 0$.
 Theorem \ref{Jacobian} generalizes the results in \cite{D1, D2, D4} on elliptic curves toward higher dimensional Abelian varieties. One may further compare our results with recent works \cite{EFM, Kim, SSW, SZ, Yo}.
 An announcement of this paper for curves appeared in \cite{D5}.

\section{Bergman kernel on hyperelliptic curves} 

The Bergman kernel, defined in \eqref{def}, is independent of the choices of the orthonormal basis for the Hilbert space of $L^2$ holomorphic top-forms.
We consider the complex analytic families of hyperelliptic curves $X_{\lambda}:=\{(x,y)\in \mathbb C^2\, |\,y^2=h_{\lambda}(x)P(x)\}$ of genus $g\geq 2$,  parameterized by $\lambda \in \Delta^*$.
 It is known (cf. \cite{Nag, We}) that the Bergman kernel $\kappa_{X_\lambda}$ on the curve $X_\lambda$ can be written as 
$$
\kappa_{X_\lambda}:=\sum_{i, j=1}^g ((\Imagine Z)^{-1})_{i, j} \,\omega_i  \otimes \overline {\omega_j},
$$ 
where $\omega_i$ is given in \eqref{basis} and $Z$ is the normalized Riemann period matrix defined by \eqref{homology basis}. 
To get the above alternative definition of $\kappa_{X_\lambda}$, it suffices to observe that by the Hodge-Riemann bilinear relation  
\begin{equation} \label{bilinear}
\frac{\sqrt{-1}}{2}\int_{X_\lambda} \omega_i \wedge \overline{ \omega_j} =  (\Imagine Z)_{ij},
\end{equation}
which is independent of the choice of the homology basis.  Since each $\omega_i$ is a linear combination of $\{e_j\}_{j=1}^g$, a complete orthonormal basis of $H_{L^2}^0(X, \mathcal K)$, the 
 Bergman kernel  defined in \eqref{def}  can be alternatively written as above in view of \eqref{bilinear}.

 \medskip
  
For our convenience, we fix once and for all  a  canonical  homology basis  $\delta_i$ and $\gamma_j$ of $H_1(X_{\lambda}, \mathbb Z)$  for each family of hyperelliptic curves defined by \eqref{X-node}, \eqref{X-cuspI} and \eqref{X-cuspII} in the following way. On $X_{\lambda}$ there is a canonical involution induced by $y \mapsto -y$, which we denote by $\sigma: X_{\lambda} \to X_{\lambda}$, and the fixed points of $\sigma$ are the Weierstra{\ss} points. For concreteness, we first deal with the  family of curves defined by \eqref{X-node}, and the  Weierstra{\ss} points  correspond to the points $x=0, x=\lambda, x=1, x=a_1, \ldots, x=a_{2g-2}$ and the point $\infty.$ We denote them by $W_1, \ldots, W_{2g+2}$, where $W_{2g+2} = \infty$.   We let $\delta_1$ be a cycle lying in a single sheet of  $X_{\lambda} $ such that it contains $0, \lambda \in \mathbb C$.
Let $\delta_2$ be a cycle lying in a single sheet of  $X_{\lambda} $ such that it contains $0, \lambda, 1, a_1 \in \mathbb C$.
  Similarly, for each $j = 3, \ldots , g$,  let $\delta_j$ be a cycle lying in a single sheet of  $X_{\lambda} $ such that it contains $0, \lambda, 1, a_1, \ldots, a_{2j-3} \in \mathbb C$. On the other hand, let $\gamma_1$ be a cycle passing from one sheet to the other as we pass the branch cut such that  it contains only two Weierstra{\ss} points   $W_2$ and $W_3$. Thus, $\gamma_1$ is made of two pieces, one in one sheet and one in the  other sheet.
In general, for each $j = 2, \ldots , g$, let $\gamma_j$ be a cycle passing from one sheet to the other as we pass the branch cut such that  it contains only two Weierstra{\ss} points  $W_{2j}$ and $W_{2j+1}$. Next, for the  family of curves defined by \eqref{X-cuspI},  the  Weierstra{\ss} points  correspond   to  the points $x=-\sqrt\lambda, x=0, x=\sqrt\lambda, x=a_1, \ldots, x=a_{2g-2}$ and the point $\infty$; for the  family of curves defined by \eqref{X-cuspII},  the  Weierstra{\ss} points  correspond  to the points $x=0,  x=\lambda, x= \lambda^2, x=a_1, \ldots, x=a_{2g-2}$ and the point $\infty.$ The cycles $\delta_i$ and $\gamma_j$
for these two families can be chosen in the same way.
Lastly, an illustration of the above choice of  $\delta_i$ and $\gamma_j$ can be found at, for example,  \cite[Figure 1.15]{CMSP}; see also \cite[IIIa--§5]{Mum} for a similar choice of homology basis.

  \medskip

By the implicit function theorem, near  $(0, 0)$ on each hyperelliptic curve $X_{\lambda}$ belonging to  either family \eqref{X-node}, \eqref{X-cuspI}, or \eqref{X-cuspII}, one may  simply take $ y$ as a   local coordinate.  To   make the computations less involved, near  $(0, 0)$ we furthermore choose the local coordinate  $z=\sqrt x$ or its opposite for a point $(x,y)$ on the first or second sheet of $X_{\lambda}$, respectively. Here $\sqrt x$ is holomorphic on   $\mathbb C \setminus L$, where $L$ is a line passing through  the distinct roots of the polynomial  $h_{\lambda}(x)P(x)$. Since for a point on either sheet it holds that $z^2=x$ and $2zdz=dx$, the globally defined Abelian differentials $\omega_i$ can be expressed  in this $z$-coordinate as below.
 Moreover,  we write $\kappa_{X_\lambda} =k_{\lambda}(z) dz \otimes  d\bar z$  and   explore the asymptotic behavior of $\psi :=\log k_{\lambda}(z)$, when the curve degenerates  as $\lambda\rightarrow 0$.

 \paragraph{Nodal family} 
A nodal Riemann surface $X$ is a connected complex space such that every point $p \in X$ has a neighborhood isomorphic to either a disk in $\mathbb C$ or $\{(x,y)\in \mathbb C^2\, |\,|x|, |y| <1\,, xy= 0\}$. In the latter, $p$ is a node. For the family of hyperelliptic curves in \eqref{X-node}, as $\lambda \to 0$, $X_{\lambda}$ degenerates to a singular curve $X_0$ with a non-separating node. 
In the above $z$-coordinate,
$$
\omega_i=\frac{2z^{2(i-1)}dz}{\sqrt {(z^2-\lambda)(z^2-1) P(z^2)}}, \, i=1,\ldots,g,
$$ 
and the Bergman kernel $\kappa_{X_\lambda}$ can be written as 
\begin{equation} \label{Bergman-g-node}  
\sum_{i,j=1}^g  ((\Imagine Z)^{-1})_{i, j}  (z^{i} \overline{z}^{j})^2 \, \frac{ 4  dz\otimes d\bar z }{ |z^4(z^2-\lambda)(z^2-1)   P(z^2)|}=:
{k}_{\lambda}(z)dz\otimes d\bar z.
\end{equation}
Particularly, if $P =(x-a)(x-b)$, then \eqref{Bergman-g-node} reduces to
\begin{equation} \label{Bergman-2-node}  4  \frac{((\Imagine Z)^{-1})_{1,1}  + ((\Imagine Z)^{-1})_{1,2} \overline z^{2 } + ((\Imagine Z)^{-1})_{2,1}  {z}^{2} + ((\Imagine Z)^{-1})_{2,2} |z|^{4}  }{ |(z^2-1)(z^2-a)(z^2-b) (z^2-\lambda)|} dz\otimes d\bar z.\end{equation}
The pinching coordinate could be used for general (possibly non-hyperelliptic) algebraic curves degenerating to a nodal curve, which represents a boundary point of the Deligne-Mumford compactification of moduli space (see \cite{F, Y, Ma, We}). 
For a compact Riemann surface of genus $g \geq 2$, its non-hyperellipticity   is characterized by the non-vanishing of the Gaussian curvature of its Bergman kernel (see \cite{Lew}). However, in this paper we do not use the effective pinching coordinate method, which works for both non-separating and separating nodal cases, since we need to deal with the cuspidal curves as well.

\paragraph{Cuspidal family, Case I} 
For the family of hyperelliptic curves in \eqref{X-cuspI}, as $\lambda \to 0$, $X_{\lambda}$ degenerates to a singular curve $X_0$ with a cusp.  Similarly,  in the above $z$-coordinate,
$$
\omega_i=\frac{2z^{2(i-1)}dz}{\sqrt {(z^4-\lambda)  P(z^2)}}, \, i=1,\ldots,g,
$$ 
and   the Bergman kernel $\kappa_{X_\lambda}$ can be written as 
\begin{equation} \label{Bergman-g-cuspI}  
\kappa_{X_\lambda}= \sum_{i,j=1}^g  ((\Imagine Z)^{-1})_{i, j}  (z^{i} \overline{z}^{j})^2 \, \frac{ 4  dz\otimes d\bar z}{ |z^4(z^4-\lambda)\cdot P(z^2)|}=:
{k}_{\lambda}(z) dz\otimes d\bar z.
\end{equation}
Particularly, if $P =(x-a)(x-b)$, then \eqref{Bergman-g-cuspI} reduces to
\begin{equation} \label{Bergman-2-cuspI}  4   \frac{((\Imagine Z)^{-1})_{1,1}  + ((\Imagine Z)^{-1})_{1,2} \overline z^{2 } + ((\Imagine Z)^{-1})_{2,1}  {z}^{2} + ((\Imagine Z)^{-1})_{2,2} |z|^{4}  }{ |(z^2-a)(z^2-b) (z^4-\lambda)|} dz\otimes d\bar z.\end{equation}

\paragraph{Cuspidal family, Case II} 
For the family of hyperelliptic curves in \eqref{X-cuspII}, as $\lambda \to 0$, $X_{\lambda}$ degenerates to a singular curve $X_0$ with a  cusp. Similarly,  in the above $z$-coordinate,
$$
\omega_i=\frac{2z^{2(i-1)}dz}{\sqrt {(z^2-\lambda)  (z^2-\lambda^2) P(z^2)}}, \, i=1,\ldots,g,
$$ 
and   the Bergman kernel $\kappa_{X_\lambda}$ can be written as 
\begin{equation} \label{Bergman-g-cuspII}  
\kappa_{X_\lambda}=\sum_{i,j=1}^g  ((\Imagine Z)^{-1})_{i, j}  (z^{i} \overline{z}^{j})^2 \, \frac{ 4 dz\otimes d\bar z }{ |z^4 (z^2-\lambda)(z^2-\lambda^2)\cdot P(z^2)|}=:
{k}_{\lambda}(z)  dz\otimes d\bar z.
\end{equation}
Particularly, if $P =(x-a)(x-b)$, then \eqref{Bergman-g-cuspII} reduces to
\begin{equation} \label{Bergman-2-cuspII}  4 \frac{((\Imagine Z)^{-1})_{1,1}  + ((\Imagine Z)^{-1})_{1,2} \overline z^{2 } + ((\Imagine Z)^{-1})_{2,1}  {z}^{2} + ((\Imagine Z)^{-1})_{2,2} |z|^{4}  }{ |(z^2-a)(z^2-b) (z^2-\lambda)(z^2-\lambda^2)|} dz\otimes d\bar z.\end{equation}

\paragraph{Normalization of singular curves}

For the {\bf nodal} singular curve
$X_{0}$ in \eqref{X-node}, we consider its normalization $Y:=\{(x,y)\in \mathbb C^2\, |\,y^2= (x-1)P(x)\}$, which is a smooth curve of genus $g-1$. 
The global Abelian differentials $\omega_i$ can be expressed in the above $z$-coordinate  as 
$$
\omega_i=\frac{x^{(i-1)}dx}{\sqrt { (x-1)P(x)}}=\frac{2z^{2i-1}dz}{\sqrt {(z^2-1)P(z^2)}}, \, i=1,\ldots,g-1.
$$
Let $Z_{0}$ be the normalized period matrix of $Y$.
By the $L^2$ removable singularity theorem, $Y$ and the regular part of $X_{0}$ have the same Bergman kernel which can be written    as 
 $\kappa_ 0=k_0(z) dz\otimes d\bar z$, where
  \begin{equation} \label{Bergman-reg-node}  
k_0(z)=\frac{4}{|z^2 (z^2-1) P(z^2)|}    \sum_{i,j=1}^{g-1}  ((\Imagine  Z_0 )^{-1})_{i, j}   ( z^{i} \overline{z}^{j})^2.
  \end{equation}

For the {\bf cuspidal} singular curve
$X_{0}:=\{(x,y)\in \mathbb C^2\, |\,y^2=x^3 P(x)\} $, its normalization is $Y:= \{(x,y)\in \mathbb C^2\, |\,y^2= x P(x)\} $. Similarly, in the above $z$-coordinate, the  Abelian differentials $\omega_i$ can be expressed as 
$$
\omega_i=\frac{x^{(i-1)}dx}{\sqrt {  x  P(x)}}=\frac{2z^{2i-2}dz}{\sqrt { P(z^2)}}, \, i=1,\ldots,g-1,
$$
and the Bergman kernel on $Y$ can be written  as 
$\kappa_ 0=k_0(z) dz\otimes d\bar z$,
where
   \begin{equation}  \label{Bergman-reg-cusp}  
k_0(z)=\frac{4}{|z^4 P(z^2)|}    \sum_{i,j=1}^{g-1}  ((\Imagine  Z_0 )^{-1})_{i, j}   ( z^{i} \overline{z}^{j})^2.
  \end{equation}

\section{Non-separating node: genus-two curves}

In this section, we consider a family of genus two curves 
\begin{equation} \label{X-node-2}
X_{\lambda}:=\{(x,y)\in \mathbb C^2\, |\,y^2=x(x-{\lambda}) (x-1)(x-a)(x-b)\}, 
\end{equation}
where $a, b, \lambda$ are distinct complex numbers satisfying $0<|\lambda|<1<|a|<|b|$. As $\lambda \to 0$, $X_{\lambda}$ degenerates to a singular curve $X_0$ with a non-separating node. The normalization of $X_0$ is an elliptic curve $\{(x,y)\in \mathbb C^2\, |\,y^2=(x-1)(x-a)(x-b) \},$  
whose period is $c$ given in \eqref{period-c}. Let 
$$
  c_1 :=   
   \int_1^a \frac{ \sqrt{ab}dx}{\sqrt{(x-1)(x-a)(x-b)}} 
  $$
be a constant depending on $a, b$. The Bergman kernels on $X_{\lambda}$ and on the normalization of $X_0$ are denoted by $\kappa_ {X_\lambda}$ and $\kappa_0$, respectively. In the local coordinate $z:=\sqrt{x}$ near $(0, 0)$, we write $\kappa_ {X_\lambda}=k_{\lambda}(z) dz\otimes d\bar z$, and $\kappa_ 0=k_0(z) dz\otimes d\bar z$. Then, our result on the asymptotic behaviour of the Bergman kernel $\kappa_{X_\lambda}$ with precise coefficients is stated as follows.

\begin{theorem} \label{potential-2-node}  
For $X_{\lambda}$ defined by \eqref{X-node-2}, as $\lambda\rightarrow 0$, it holds that
\item [(i)] $ \kappa_ {X_\lambda} \to \kappa_0$;

\medskip

\item [(ii)] for small $|z| \neq 0$ , 
$$
\log {k}_{\lambda}(z) -\log {k}_0(z) 
 \sim  \frac{\pi}{-\log |\lambda|} {\frac{  \Imagine c  +  ( {z}^{2}+\overline z^{2 }) \Real \{c_1^{-1}\}  } {   |z|^{4}   }      }   .
$$ 
  \end{theorem}
  
In fact, the results in Section 4 on general hyperelliptic curves largely rely on our proofs in this section. To prove Theorem \ref{potential-2-node}, we need the following two lemmata by analyzing the asymptotics of the $A$-period matrix and $B$-period matrix on $X_{\lambda}$.

\begin{lemma} \label{period-A-2} 
Under the same assumptions as in Theorem \ref{potential-2-node}, as $\lambda\rightarrow 0$, it holds that
  $$ 
  A = \left(
    \begin{array}{cc}
      \frac{-2\pi}{\sqrt{ab}}  & c_2 \\
     0   & \frac{-2c_1}{\sqrt{ab}} 
    \end{array}
  \right) + \Hi(\lambda) ,
  $$ 
  where 
$c_2:= \int_{\delta_2} \frac{dx}{x\sqrt{(x-1)(x-a)(x-b)}}$ and $\lim\limits_{\lambda \to 0} \frac{\Hi( \lambda)}{\lambda}$ is a finite matrix.

\end{lemma}

  \begin{proof} To estimate the four entries of $A$ one by one, we use the choice of cycles $\delta_1, \delta_2$ specified in Section 2. Firstly, $A_{1,1}=\int_{\delta_1}\omega_1$, where $\delta_1$ only contains $0$ and $\lambda$.  By the substitutions $t={\lambda}^{-1}$ and $s= {x}^{-1}$, we get the dual cycle $\tilde \delta_1$ which contains $\{\infty, t \}$ so that $-\tilde\delta_1$ contains $\{0, 1, {a}^{-1}, {b}^{-1}\}$, for $|s| \in (1, |t|)$. As $\lambda\rightarrow 0$,
\begin{align*}
A_{11}& = \int_{\tilde\delta_1} \frac{-s^{-2}ds}{\sqrt{ s^{-1}(s^{-1} - t^{-1})(s^{-1}-1)(s^{-1}-a)(s^{-1}-b)}}\\
&=-\int_{-\tilde\delta_1} \frac{-\sqrt{s} \sqrt{t} ds}{\sqrt{ab(s-1)(s- a^{-1})(s- b^{-1})(s-t) }}\\
&=\int_{-\tilde\delta_1} \frac{ \sqrt{s}  ds}{ \sqrt{-ab(s-1)(s-a^{-1})(s-b^{-1})}} \left( 1+\frac{s}{2t}+\Hi\left(\frac{s^2}{{t}^2}\right)\right) \\
&=\int_{-\tilde\delta_1} \frac{ ds}{ -s\sqrt{-ab}} \left( 1+\frac{s}{2t}+\Hi \left(\frac{s^2}{{t}^2}\right)\right) \left(1+\Hi({{s^{-1}}})\right)\left(1+\Hi({{s^{-1}}})\right)\left(1+\Hi({{s^{-1}}})\right)\\
&=\int_{-\tilde\delta_1} \frac{ ds}{ -s\sqrt{-ab}} \left( 1+\frac{s}{2t}+\Hi(\frac{s^2}{{t}^2})\right) \left(1 +\Hi({{s}^{-1}})\right)\\
&=\frac{ 1}{ -\sqrt{-ab}}\int_{-\tilde\delta_1}  \left(1+\Hi(t^{-1})\right)\frac{ds}{s}\\
&=\frac{ 2\pi }{ -\sqrt{ab}} \left(1 +\Hi({\lambda})\right).
\end{align*} 
Notice that we have used the Maclaurin expansion 
\begin{equation}
\frac{1}{\sqrt{s-a}}=
\begin{cases}
 {\sqrt{-a}^{-1}}  \left( 1+ \Hi\left( {s }{{a^{-1}} }\right)\right), &  |s|<|a|;\\
\sqrt{s}^{-1}  \left(1+ \Hi({ {s^{-1} a}}) \right),  &    |s|>|a|.
\end{cases}
\end{equation}
Secondly, look at $A_{21}=\int_{\delta_1}\omega_2$ and similarly it holds that \begin{align*}
A_{21}&=\int_{-\tilde\delta_1} \frac{ ds}{ -s^2\sqrt{-ab}} \left( 1+\frac{s}{2t}+\Hi\left (\frac{s^2}{{t}^2}\right)\right) \left(1 +\Hi(s^{-1})\right)\\
&=\frac{ 1}{ -\sqrt{-ab}}\int_{-\tilde\delta_1}  \left(\frac{1}{2t}+\Hi(t^{-2})\right)\frac{ds}{s}\\
&=\frac{ 2\pi }{ -\sqrt{ab}} \left(\frac{\lambda}{2} +\Hi(\lambda^2)\right).
\end{align*} 
Thirdly, since  $\delta_2$ contains $\{0, \lambda, 1, a\}$, as $\lambda \to 0$, it holds that
\begin{align*}
A_{12}=& \int_{\delta_2} \frac{dx}{x\sqrt{(x-1)(x-a)(x-b)}} \left(1+\frac{\lambda}{2x} +\Hi\left (\frac{\lambda^2}{x^2}\right)\right)\\
= & \int_{\delta_2} \frac{dx}{x\sqrt{(x-1)(x-a)(x-b)}} \left(1+\Hi\left (\lambda \right)\right).
\end{align*} 
Lastly,
$$
A_{22}=\int_{\delta_2} \frac{dx}{\sqrt{(x-1)(x-a)(x-b)}}\left(1+\Hi(\lambda)\right)=-2\int_1^a \frac{dx}{\sqrt{(x-1)(x-a)(x-b)}}\left(1+\Hi(\lambda)\right).
$$

\end{proof}

\begin{lemma} \label{period-B-2} Under the same assumptions as in Theorem \ref{potential-2-node}, as $\lambda\rightarrow 0$, it holds that $$
  B\sim \left(
    \begin{array}{cc}
      \frac{-2\log \lambda }{\sqrt{-ab}}   & d_1 \\
      \frac{-2}{ \sqrt{-ab}} & d_2
    \end{array}
  \right),
$$ where $d_1:=-2\int_a^b\frac{dx}{x\sqrt{(x-1)(x-a)(x-b)}}$ and $d_2:=-2\int_a^b\frac{dx}{\sqrt{(x-1)(x-a)(x-b)}}.$ \end{lemma}
 
  \begin{proof} 
As $t \to \infty$, we make use of the following computations (cf. \cite{CMSP}).

\begin{equation} \label{I} \int_1^t\frac{  ds}{s\sqrt{s-t}}=\left.\frac{-2}{\sqrt t}\sqrt{-1} \log\left(\sqrt{\frac{t}{s}}+\sqrt{\frac{t}{s}-1}\right)\right|_1^t \sim \frac{\sqrt{-1}}{\sqrt t} \log t.\end{equation}

\begin{equation} \label{II} 
\int_1^t\frac{  ds}{s^2\sqrt{s-t}}=\left.\frac{\sqrt{s-t}}{ts}\right|_1^t +\frac{1}{2t} \int_1^t\frac{  ds}{s\sqrt{s-t}}= \frac{-\sqrt{1-t}}{t}+\frac{\sqrt{-1}}{2t\sqrt t} \log t \sim  -\frac{\sqrt{-1}}{\sqrt{t}}.
\end{equation}
In particular, \eqref{II} yields the boundedness of $$\int_1^t\frac{ \sqrt{t} }{s\sqrt{ s-t} }  \Hi(s^{-1})ds.$$
More generally, for any integer $\alpha \geq 1$, as $t \to \infty$, it holds that
\begin{equation} \label{alpha}
\int_1^{t}\frac{ ds }{s^{\alpha +1} \sqrt{s-t}}  = \left.\frac{\sqrt{s-t}}{\alpha ts^\alpha}\right|_1^t +\frac{2\alpha -1}{2\alpha t} \int_1^t\frac{  ds}{s^\alpha\sqrt{s-t}} \sim  -\frac{\sqrt{-1}}{\alpha \sqrt{t}}.
\end{equation}
Similar to the proof of Lemma \ref{period-A-2},  we use the choice of cycles $\gamma_1, \gamma_2$ specified in Section 2 to estimate the four entries of $B$ one by one.
 Firstly, by \eqref{I} and \eqref{II}, for $|s| \in (1, |t|)$,
\begin{align*} 
B_{11}=&-2\int_{\lambda}^1\frac{dx}{\sqrt{x(x-\lambda)(x-1)(x-a)(x-b)}} \\
=&-2\int_t^1 \frac{-\sqrt{s} \sqrt{t} ds}{\sqrt{ab(s-1)(s- a^{-1})(s- b^{-1})(s-t) }}\\
= &-2\int_1^t\frac{\sqrt{s} \sqrt{t} ds}{\sqrt{ (s-t)ab}} \frac{1}{s\sqrt{s}}  \left(1+\Hi({{s^{-1}}})\right)\left(1+\Hi({{s^{-1}}})\right)\left(1+\Hi({{s^{-1}}})\right)\\
= & -\frac{2\sqrt{t}}{\sqrt{ab}}\int_1^t\frac{  ds}{s\sqrt{ s-t} }  \left(1+\Hi({{s^{-1}}})\right) \\
\sim & \frac{-2\log \lambda }{\sqrt{-ab}},
\end{align*} 
as $\lambda \to 0$.
Secondly, by \eqref{II}, 
$$
B_{21}=  -\frac{2\sqrt{t}}{\sqrt{ab}}\int_1^t\frac{  ds}{s^2\sqrt{ s-t} }  \left(1+\Hi({{s^{-1}}})\right) \sim\frac{-2}{ \sqrt{-ab}}.
$$
Finally, similar to $A_{12}$,
$$
B_{12} \sim -2 \int_a^b \frac{dx}{x\sqrt{(x-1)(x-a)(x-b)}},
$$
and
$$
B_{22} \sim -2 \int_a^b \frac{dx}{\sqrt{(x-1)(x-a)(x-b)}}.
$$

\end{proof}

Combining Lemmata \ref {period-A-2} and \ref {period-B-2} with \eqref{Bergman-2-node}, we get the asymptotics of the Bergman kernels.

 \begin{proof} [\bf Proof of Theorem \ref{potential-2-node}] 
Notice that the period is defined as
  \begin{equation} \label{period-c}
  c :=  \frac{\int_a^b\frac{dx}{\sqrt{(x-1)(x-a)(x-b)}}} {\int_1^a \frac{dx}{\sqrt{(x-1)(x-a)(x-b)}}}  =   \tau\left(\frac{1-b}{1-a}\right),
  \end{equation} 
  where $\tau(\cdot)$ is the inverse of the elliptic modular lambda function. 
On the normalization of the nodal curve $X_0$, by \eqref{Bergman-reg-node}, in the local coordinate $z=\sqrt x$, the Bergman kernel is exactly 
$\kappa_ 0=k_0(z) dz\otimes d\bar z$, where
  \begin{equation}  \label{reg-2-node}
k_0(z)=\frac{ 4 |z |^2  }{  (\Imagine c) |(z^2-1)   (z^2-a)(z^2-b)|}.
  \end{equation} 
By Lemma \ref{period-A-2}, as $\lambda\to 0$,
   $$
   A^{-1}
 \sim \left(
    \begin{array}{cc}
      \frac{-2\pi}{\sqrt{ab}} & c_2 \\
        \Hi(\lambda)  & \frac{-2c_1}{\sqrt{ab}}  
    \end{array}
  \right)^{-1} = \left(
    \begin{array}{cc}
      \frac{-\sqrt{ab}}{2\pi}  & \frac{ \sqrt{ab} c_2  } { {2\pi  d_2 } } \\
     \Hi(\lambda) & \frac{\sqrt{ab}} {-2c_1}    
    \end{array}
  \right).
  $$
   Let $Z=A^{-1}B$ denote the normalized period matrix of $X _{\lambda}$. Then, it follows that $$
Z \sim 
  \left(
    \begin{array}{cc}
       \frac{-\sqrt{-1}}{\pi} \log \lambda    
       &  \frac{  ab c_2 d_2 + \sqrt{ab}{2c_1 d_1}    }{   {-4\pi c_1}  }   \\
    { -\sqrt{-1}}  {  c_1}^{-1}     &  c
    \end{array}
  \right), \quad \Imagine Z\sim \left(
    \begin{array}{cc}
      -\frac{\log |\lambda|}{\pi }  &   -\Real \{c_1^{-1}\}  \\
-\Real \{c_1^{-1}\}  & \Imagine c
    \end{array}
  \right).
$$ 
So, as $\lambda\rightarrow 0$, it holds that
$$ 
(\Imagine Z)^{-1}\sim    \left(
    \begin{array}{cc}
   \frac{\pi}{-     \log |\lambda|  }          & \frac{ -\Real \{c_1^{-1}\} \pi  }{   ( \log |\lambda|) \Imagine c }  \\
     \frac{ -\Real \{c_1^{-1}\} \pi  }{   ( \log |\lambda|) \Imagine c }  &      ({\Imagine c })^{-1}
    \end{array}
  \right).
$$ 
By \eqref{Bergman-2-node} and \eqref{reg-2-node},
$\kappa_{X_\lambda} \to \kappa_0.
$ Moreover, in the local coordinate $z=\sqrt x$ near $(0, 0)$,
$$
{k}_{\lambda}(z) - {k}_{0}(z)  \sim \frac{4 \pi}{ |(z^2-1)(z^2-a)(z^2-b) (z^2-\lambda)|}   \frac{  1  +  \frac{\Real \{c_1^{-1}\}}{ \Imagine c } ({z}^{2}+\overline z^{2 })     }{ - \log |\lambda|},
  $$
    which yields the conclusion.
    
 \end{proof}

Since the moduli space of genus-two curves is 3 dimensional, one may consider the more general family $X_{\lambda, a,b}:=\{y^2=x(x-1)(x-\lambda)(x-a)(x-b))\}$, parameterized by three distinct complex numbers $\lambda, a, b \in \mathbb C\setminus \{0,1\}$. As $\lambda, a$ or $b$ tends to $0, 1$ or $\infty$, or towards one another, $X_{\lambda, a, b}$ will become singular. In our setting, we fix the other two parameters $a$ and $b$, and move $\lambda$ only, so the precise asymptotic coefficients we have obtained depend on both $a$ and $b$.

\section{Non-separating node: hyperelliptic and general curves} 
This section is devoted to the proof of Theorem \ref{potential-g-node}, as a generalization of Theorem \ref{potential-2-node} towards the hyperelliptic case. When $\lambda\in \mathbb C\setminus \{0, 1, a_1, \ldots, a_{2g-2}\}$, $X_{\lambda}$ defined in \eqref{X-node} has genus $g$. To prove Theorem \ref{potential-g-node}, we use the choice of cycles $\delta_j, \gamma_j$ specified in Section 2, and need to analyze on $X_{\lambda}$ the asymptotics of its $A$-period matrix and $B$-period matrix, denoted by $A$ and $B$, respectively. Meanwhile, for the normalization $Y:=\{(x,y)\in \mathbb C^2\, |\,y^2=(x-1)P(x)\} $ of genus $g-1$, denote its $A$-period matrix and $B$-period matrix by $A_0$ and $B_0$, respectively. 

\begin{lemma} \label{period-A-hyp-node} Under the same assumptions as in Theorem \ref{potential-g-node}, as $\lambda\rightarrow 0$, it holds that
$$ 
A \sim \left(
    \begin{array}{cccccc}
     \frac{2\pi}{\sqrt{P(0) }} & \Hi(1) \\
* &      A_0\\
    \end{array}
  \right),
$$
where $*=(\Hi( \lambda), 0,\ldots,0)^T$ is a column vector with $g-1$ rows.
\end{lemma}
 
  \begin{proof} 
  
  Firstly, $A_{11}=\int_{\delta_1}\omega_1$, where $\delta_1$ only contains $0$ and $\lambda$. 
  By the substitutions $t={\lambda}^{-1}$ and $s= {x}^{-1}$,   we get the dual cycle $\tilde \delta_1$ which contains $\{\infty, t \}$ so that $-\tilde\delta_1$ contains $\{0, 1, {a_1}^{-1}, \ldots,  a_{2g-2}^{-1} \}$   for $|s| \in (1, |t|)$. As $\lambda\rightarrow 0$,
\begin{align*}
A_{11} & =
\int_{-\tilde\delta_1} \frac{ ds}{ \sqrt{-1} s\sqrt{P(0)}} \left( 1+ \Hi \left(\frac{s}{{t}}\right)\right) \left(1+\Hi({{s^{-1}}})\right) \\
& \sim  \int_{-\tilde\delta_1} \frac{   ds}{ \sqrt{-1} s\sqrt{P(0)}}  \left(1+\Hi({{s^{-1}}})\right) \\
& = \int_{-\tilde\delta_1} \frac{ ds}{ \sqrt{-1} s\sqrt{P(0)}} \\
&= \frac{2\pi  }{  \sqrt{P(0)}}.
\end{align*}

\begin{align*}
A_{21}&  =\int_{-\tilde\delta_1} \frac{  ds}{ \sqrt{-1} s^{2} \sqrt{P(0)}} \left( 1+ \Hi \left(\frac{s}{{t}}\right)\right) \left(1+\Hi({{s^{-1}}})\right) \\
&= \int_{-\tilde\delta_1} \frac{ ds}{ \sqrt{-1} \sqrt{P(0)}} \left( {{s^{-2}}} + \Hi \left(\frac{1}{{ts^{i-1}}}\right)\right) \left(1+\Hi({{s^{-1}}})\right)\\
&= \int_{-\tilde\delta_1} \frac{ ds}{ \sqrt{-1} \sqrt{P(0)}} \left(  + \Hi \left(\frac{1}{{ts}}\right)\right)\\
& =\Hi(\lambda).
\end{align*} 
For $3\leq i\leq g$,
\begin{align*}
A_{i1}&  =\int_{-\tilde\delta_1} \frac{  ds}{ \sqrt{-1} s^{i} \sqrt{P(0)}} \left( 1+ \Hi \left(\frac{s}{{t}}\right)\right) \left(1+\Hi({{s^{-1}}})\right) \\
&= \int_{-\tilde\delta_1} \frac{ ds}{ \sqrt{-1} \sqrt{P(0)}} \left( {{s^{-i}}} + \Hi \left(\frac{1}{{ts^{i-1}}}\right)\right) \left(1+\Hi({{s^{-1}}})\right)\\
&=0.
\end{align*} 
 Secondly, since $\delta_2$ contains only $0, \lambda, 1$ and $a_1$ (but $a_2,\ldots, a_{2g-2}$), it holds that
 \begin{align*}
A_{12}&= \int_{\delta_2} \frac{dx}{\sqrt{x(x-{\lambda})(x-1)P(x)}}\\
&=\int_{\delta_2} \frac{dx}{x \sqrt{(x-1) P(x) }} \left( 1 +\Hi \left(\frac{{\lambda} }{{x}}\right)\right) \\
&\sim \int_{\delta_2} \frac{dx}{x \sqrt{(x-1) P(x) }}.
\end{align*}  
In general, for $A_{i2} $, there is an extra $x^{i-1}$ in the original integrand above, so
$$
A_{i2} \sim \int_{\delta_2} \frac{x^{i-2} dx}{  \sqrt{(x-1) P(x) }}.
$$
Thirdly,   $\delta_j$ contains only $0, \lambda^2, \lambda, a_1, \ldots, a_{2j-3}$ (but $a_{2j-2}, \ldots, a_{2g-2}$). 
 In general, for $j\geq 2$, $A_{ij}$ is asymptotic to the same integrand along $\delta_j$ instead of along $\delta_2$, i.e., 
$$
A_{ij} \sim \int_{\delta_j} \frac{x^{i-2} dx}{  \sqrt{(x-1) P(x) }},
$$
which is exactly the same as the corresponding entry of $A_0$ when $i\geq 2$.

\end{proof}

\begin{lemma} Under the same assumptions as in Theorem \ref{potential-g-node}, as $\lambda\rightarrow 0$, it holds that
$$ B \sim \left(
    \begin{array}{cccccc}
    \frac{-2\sqrt{-1} }{\sqrt{P(0)}}  { \log \lambda}  & \Hi(1) \\
    \star & B_0 
        \end{array}
  \right),
$$
where $\star$ is a column vector whose entries are all $\frac{-2\sqrt{-1} }{\sqrt{P(0)}}$ with $g-1$ rows.
\end{lemma}

  \begin{proof} 
  
  Firstly, by \eqref{I} and \eqref{II}, for $|s| \in (1, |t|)$,
      \begin{align*}
B_{11} &=-2\int_{\lambda }^1 \frac{dx}{\sqrt{x(x-{\lambda})(x-1)\cdot P(x)}} \\
& =-2\int_{\lambda}^1 \frac{dx}{\sqrt{x(x-{\lambda})}}\frac{1}{\sqrt{-P(0)}}(1+\Hi(x))\\
& = \frac{2\sqrt{t}}{\sqrt{P(0)}}\int_{1}^t \frac{ds}{s\sqrt{ s-t}} (1+\Hi(s^{-1})) \\
& \sim \frac{-2 }{\sqrt{P(0)}}  {\sqrt{-1}}  \log \lambda,
\end{align*}
where
the last equality holds due to  the substitutions $t={\lambda}^{-1}$ and $s= {x}^{-1}$.
Thus, for $2\leq i\leq g$, by \eqref{alpha},
$$
B_{i1}=-2\int_{\lambda }^1 \frac{x^{i-1}dx}{\sqrt{x(x-{\lambda})(x-1)P(x)}}= \frac{2\sqrt{t}}{\sqrt{P(0)}}\int_{1}^t \frac{ds}{s^i\sqrt{ s-t}} (1+\Hi(s^{-1}))\sim \frac{-2\sqrt{-1} }{\sqrt{P(0)}}.     
$$
For Column $j$, $2 \leq j \leq g$, we use Taylor series expansion of $\sqrt{x-{\lambda}}^{-1}$ to get that
 $$
B_{ij} \sim -2\int_{a_{2j-3}}^{a_{2j-2}} \frac{x^{i-2} dx}{  \sqrt{(x-1) P(x) }},
$$
which is exactly the same as the corresponding entry of $B_0$ when $i\geq 2$.

\end{proof}

Now we will give a proof of Theorem \ref{potential-g-node}.

\begin{proof} [\bf Proof of Theorem \ref{potential-g-node}]

 By Lemma \ref{period-A-hyp-node} and the block matrix inversion, we know that
$$ 
A^{-1} \sim \left(
    \begin{array}{cccccc}
   \frac {\sqrt{P(0) }}{2\pi}  & \Hi(1)  \\
     \Hi( \lambda) & A_0^{-1}
    \end{array}
  \right),
$$
where both $(\Hi(1))^T$ and $\lim\limits_{\lambda \to 0} \frac{\Hi( \lambda)}{\lambda}$ are finite column vectors with $g-1$ rows.
Therefore,  as $\lambda\to 0$,
$$
Z=A^{-1}B \sim  \left(
    \begin{array}{cccc}
   \frac{-\sqrt{-1} }{\pi}  { \log \lambda}    & \Hi( 1)\\
      A_0^{-1} \star & A_0^{-1} B_0
    \end{array}
  \right), 
  \quad 
  \Imagine Z \sim  \left(
    \begin{array}{cccc}
\frac{\log |\lambda|}  {-\pi}  & \Hi( 1) \\
     \Imagine (A_0^{-1} \star) &  \Imagine  Z_0
     \end{array}
  \right),
  $$
  and
  $$ (\Imagine Z)^{-1}\sim  \left(
    \begin{array}{cccc}
 \frac{-\pi}  {\log |\lambda|}    & \Hi (({\log |\lambda|})^{-1})     \\
   (\Imagine  Z_0 )^{-1}   \Imagine (A_0^{-1} \star)    \frac{\pi}  {\log |\lambda|}  &  (\Imagine  Z_0 )^{-1}   
    \end{array}
  \right).
  $$
  In fact, since $Z$ is symmetric, the off-diagonal block matrices in each matrix above concerning $Z$ are the transpose of each other.
By \eqref{Bergman-g-node} and \eqref{Bergman-reg-node}, as $\lambda \to 0$,
$\kappa_ {X_\lambda} \to \kappa_0.
$ Moreover, in the local coordinate $z=\sqrt x$ near $(0, 0)$, it holds that
$$
{k}_{\lambda}(z) - {k}_{0}(z) \sim
  \frac{4\pi}{|(z^2-\lambda)(z^2-1) P(z^2)|}  \frac{ 1- 2 \Real  \sum\limits_{i=1}^{g-1}   \left((\Imagine  Z_0 )^{-1}   \Imagine (A_0^{-1} \star)  \right)_{i} {z}^{2i}}{ - \log |\lambda|} 
  $$
  which yields the conclusion.

 \end{proof}

If we ignore the precise asymptotic coefficients, then the leading term growth in Theorem \ref{potential-g-node} corresponds to \cite [Proposition 3.2] {HJ}. The degeneration of  nodal non-hyperelliptic curves was also treated in \cite {HJ}\footnote{The author is grateful to Professor Z. Huang for bringing attention the paper \cite{HJ} during the 2016 Tsinghua Sanya International Mathematics Forum, where a preliminary version of this work was presented.}.

\section{Cusp I: genus-two curves}

In this section, we consider a family of genus two curves 
\begin{equation} \label{X-cuspI-2}
X_{\lambda}:=\{(x,y)\in \mathbb C^2\, |\,y^2=x(x^2-{\lambda}) (x-a)(x-b)\}, 
\end{equation}
where $a, b, \lambda$ are distinct complex numbers satisfying $0<|\lambda| <|a|<|b|$. As $\lambda \to 0$, $X_{\lambda}$ degenerates to a singular curve $X_0$ with an ordinary  cusp. The normalization of $X_0$ is an elliptic curve $\{(x,y)\in \mathbb C^2\, |\,y^2=x(x-a)(x-b) \} $, whose period is $\tau$ given in \eqref{period-tau}. Let 
$$
c_4:=-2\int_0^1\frac{ (x-1)dx}{ \sqrt{x(x-1)(x-2)}}
$$
 and let $$
  c_3 :=   
   \int_0^a \frac{ \sqrt{ab}dx}{\sqrt{x(x-a)(x-b)}} , \quad 
  $$
be a constant  depending on $a, b$. The Bergman kernels on $X_{\lambda}$ and on the normalization of $X_0$ are denoted by $\kappa_ {X_\lambda}$ and $\kappa_0$, respectively. In the local coordinate $z:=\sqrt{x}$ near $(0, 0)$, we write $\kappa_ {X_\lambda}=k_{\lambda}(z) dz\otimes d\bar z$, and $\kappa_ 0=k_0(z) dz\otimes d\bar z$. Then, our result on the asymptotic behaviour of the Bergman kernel $\kappa_{X_\lambda}$ with precise coefficients is stated as follows.

\begin{theorem} \label{potential-2-cuspI}  For $X_{\lambda}$ defined by \eqref{X-cuspI-2}, as $\lambda\rightarrow 0$, for small $|z| \neq 0$, it holds that
\item [(i)] 
$$
  {k}_{\lambda}(z) \to  {k}_0(z)    \left(\frac{\Imagine \tau     }{ |z^4  |}   +1\right), \quad i.e., \,
\kappa_{X_\lambda}  \not \to \kappa_0;
  $$
\item [(ii)] 
  $$
\log {k}_{\lambda}(z) - \log {k}_0(z)  -\log \left(\frac{\Imagine \tau     }{ |z^4  |}   +1\right)   \sim    \frac{   \Real \left\{   \lambda^{1/4} \frac{ c_4 } {  c_3   }  \right\}  (z^{2 } +\overline {z}^{2})}{ |z|^{4}  +{\Imagine \tau }   },
$$
where $\tau$ is the period (scalar) of the elliptic curve $\{y^2=x(x-a)(x-b)\}$.
 \end{theorem}

To prove Theorem \ref{potential-2-cuspI}, we need the following two lemmata by analyzing the asymptotics of the $A$-period matrix and $B$-period matrix on $X_\lambda$. Our choice of cycles $\delta_j, \gamma_j$ is specified in Section 2.

\begin{lemma} \label{period-A-2-cuspI} Under the same assumptions as in Theorem \ref{potential-2-cuspI}, as $\lambda\rightarrow 0$, it holds that

  $$ 
  A \sim \left(
    \begin{array}{cc}
\frac{c_5}{-\sqrt{ab}} \lambda^{-1/4} & c_6 \\
      \frac{ c_4 }{-\sqrt{ab}} \lambda^{1/4} & \frac{-2c_3}{\sqrt{ab}}
    \end{array}
  \right),
  $$ 
  where $c_5:= -2 \int_0^1\frac{du}{ \sqrt{u(u-1)(u-2)}}$ and $c_6$ depends on $a, b$.
\end{lemma}

\begin{proof} We estimate the four entries one by one. Firstly, let $\delta_1$ contain only $-\sqrt{\lambda}$ and $0$. By the Cauchy Integral Theorem and Taylor series expansion, we know that 
\begin{align*} A_{11}&=-2\int_{-\sqrt{\lambda}}^0\frac{dx}{\sqrt{x(x-\sqrt{\lambda})(x+\sqrt{\lambda})(x-a)(x-b)}}\\
&=\frac{-2}{-\sqrt{ab}}\int_{-\sqrt{\lambda}}^0\frac{dx}{\sqrt{x(x-\sqrt{\lambda})(x+\sqrt{\lambda})}}\left(1+\Hi(x)\right)\\
&=
 \frac{-2\lambda^{-1/4}}{-\sqrt{ab }}\int_0^1\frac{du}{ \sqrt{u(u-1)(u-2)}}\left(1+\Hi\left(\sqrt{\lambda}(u-1)\right)\right)\\
&\sim \frac{-2\lambda^{-1/4}}{-\sqrt{ab}}\int_0^1\frac{du}{ \sqrt{u(u-1)(u-2)}}=:\frac{c_5}{-\sqrt{ab }} \lambda^{-1/4},
\end{align*}
where the third equality holds due to the 
substitution
$x=(u-1) \sqrt{\lambda}$.
Secondly,  
\begin{align*}
A_{21}&= \frac{-2\lambda^{1/4}}{-\sqrt{ab }}\int_0^1\frac{(u-1)du}{ \sqrt{u(u-1)(u-2)}}\left(1+\Hi\left(\sqrt{\lambda}(u-1)\right)\right)\\
&\sim  \frac{-2 \lambda^{1/4}}{-\sqrt{ab}}\int_0^1\frac{(u-1)du}{ \sqrt{u(u-1)(u-2)}}:= \frac{ c_4 }{-\sqrt{ab}}\lambda^{1/4} .
\end{align*} 
Thirdly, let $\delta_2$ contain only $-\sqrt{\lambda}, 0, \sqrt{\lambda}$ and $a$ (but $b$). Then, 
\begin{align*}
A_{12}=& 
 \int_{\delta_2} \frac{dx}{\sqrt{x(x-a)(x-b)}} \frac{1}{\sqrt{x}} \left( 1+\Hi \left( \frac{\sqrt{\lambda}}{x}\right)\right) \frac{1}{\sqrt{x}} \left( 1+\Hi \left( \frac{\sqrt{\lambda}}{x}\right)\right)\\
\sim& \int_{\delta_2} \frac{dx}{x\sqrt{x(x-a)(x-b)}}=:c_6.
\end{align*}   
Lastly, 
\begin{align*}
A_{22}&= \int_{\delta_2} \frac{dx}{\sqrt{x(x-a)(x-b)}}\left( 1+\Hi \left( \frac{\sqrt{\lambda}}{x}\right)\right)\\
&\sim \int_{\delta_2} \frac{dx}{\sqrt{x(x-a)(x-b)}}\\
& =-2\int_{0}^a \frac{dx}{\sqrt{x(x-a)(x-b)}}=\frac{-2c_3}{\sqrt{ab}}.
\end{align*}   

\end{proof}

\begin{lemma} \label{period-B-2-cuspI} Under the same assumptions as in Lemma \ref{period-A-2-cuspI}, as $\lambda\rightarrow 0$, it holds that $$
  B \sim \frac{1}{\sqrt{-ab}}\left(
    \begin{array}{cc}
    c_5 \lambda^{-1/4}   & d_3 \sqrt{-ab}\\
    -c_4  \lambda^{1/4}  & d_4 \sqrt{-ab}
    \end{array}
  \right),
  $$ 
  where $d_3:=-2\int_a^b\frac{dx}{x\sqrt{x(x-a)(x-b)}}$ and $d_4:=-2\int_a^b\frac{dx}{\sqrt{x(x-a)(x-b)}}$. \end{lemma}

\begin{proof}  Again, we estimate all the four entries one by one. Firstly, let $\gamma_1$ contain only $0$ and $\sqrt{\lambda}$. By Cauchy Integral Theorem, we can get that 
\begin{align*} 
B_{11}&
=-2\int_0^{\sqrt{\lambda}}\frac{dx}{\sqrt{x(x^2-{\lambda})(x-a)(x-b)}}\\
&=\frac{-2}{-\sqrt{ab}}\int_0^{\sqrt{\lambda}}\frac{dx}{\sqrt{x(x-\sqrt{\lambda})(x+\sqrt{\lambda})}}\left(1+\Hi(x)\right)\\
& =
 \frac{2\lambda^{-1/4}}{\sqrt{ab}}\int_{1}^0\frac{du}{ \sqrt{-u(u-1)(u-2)}}\left(1+\Hi((-u+1)\cdot \sqrt{\lambda})\right)\\
&\sim \frac{2\lambda^{-1/4}}{\sqrt{ab }}\int_{0}^1\frac{\sqrt{-1}du}{ \sqrt{u(u-1)(u-2)}}=:\frac{ c_5}{\sqrt{-ab}}\lambda^{-1/4},\end{align*}
where the third equality holds due to the 
substitution
$
x= (1-u) \sqrt{\lambda}.
$
Secondly,
\begin{align*}
B_{21}&=\frac{2\lambda^{-1/4}}{\sqrt{ab}}\int_{1}^0\frac{(-u+1)\sqrt{\lambda}du}{ \sqrt{-u(u-1)(u-2)}}\left(1+\Hi((-u+1)\cdot \sqrt{\lambda})\right)\\
&\sim  \frac{2\lambda^{1/4}}{\sqrt{ab}}\int_{0}^1\frac{(u-1)du}{ \sqrt{-u(u-1)(u-2)}}=:\frac{\sqrt{-1} \lambda^{1/4}  c_4}{\sqrt{ab}}.
\end{align*}
Thirdly, let $\gamma_2$ contain only $a$ and $b$. Then, it holds that \begin{align*}
B_{12}=& \int_{\gamma_2} \frac{dx}{\sqrt{x(x-\sqrt{\lambda}) (x+\sqrt{\lambda})(x-a)(x-b)}}\\
=&-2 \int_a^b \frac{dx}{\sqrt{x(x-a)(x-b)}} \frac{1}{\sqrt{x}} \left( 1+\Hi\left(\frac{\sqrt{\lambda}}{2x}\right)\right) \frac{1}{\sqrt{x}} \left( 1+\Hi\left(\frac{\sqrt{\lambda}}{2x}\right)\right)\\
\sim &  -2\int_a^b \frac{dx}{x\sqrt{x(x-a)(x-b)}} =:d_3.\end{align*}  
Lastly,
$$
B_{22}= \int_{\gamma_2}  \frac{dx}{\sqrt{x(x-a)(x-b)}}  \left( 1+\Hi\left(\frac{\sqrt{\lambda}}{2x}\right)\right) \sim  -2 \int_a^b \frac{dx}{\sqrt{x(x-a)(x-b)}}=:d_4.
$$
 
\end{proof}

Combining Lemmata \ref {period-A-2-cuspI} and \ref {period-B-2-cuspI} with \eqref{Bergman-2-cuspI}, we get the asymptotics of the Bergman kernels.

\begin{proof} [\bf Proof of Theorem \ref{potential-2-cuspI}] 
Notice that the period is defined as
  \begin{equation} \label{period-tau}
 \tau :=   \left\{ \frac{\int_a^b\frac{dx}{\sqrt{x(x-a)(x-b)}}} { \int_{0}^a \frac{dx}{\sqrt{x(x-a)(x-b)}} }  \right\}=  \tau\left(\frac{b}{a}\right),
  \end{equation} 
  where $\tau(\cdot)$ is the inverse of the elliptic modular lambda function. 
On the regular part of the cuspidal curve $X_0$, by \eqref{Bergman-reg-cusp}, in the local coordinate $z=\sqrt x$, the Bergman kernel is exactly 
$\kappa_ 0=k_0(z) dz\otimes d\bar z$, where
  \begin{equation}  \label{reg-2-cusp}
k_0(z)=\frac{4 }{  \Imagine  \tau \cdot \left|(z^2-a)(z^2-b)\right|}.
  \end{equation} 
By Lemma \ref{period-A-2-cuspI}, as $\lambda\to 0$,
$$
A^{-1}\sim   \left(
    \begin{array}{cc}
\frac{c_5}{-\sqrt{ab}} \lambda^{-1/4} & c_6 \\
      \frac{ c_4 }{-\sqrt{ab}} \lambda^{1/4} &  \frac{-2c_3}{\sqrt{ab}}
    \end{array}
    \right)^{-1} \sim \frac{ ab\lambda^{1/4}} { { 2c_3}  c_5} \left(
    \begin{array}{cc}
     \frac{-2c_3}{\sqrt{ab}} & -c_6 \\
       \frac{ c_4}{\sqrt{ab}} \lambda^{1/4} & \frac{c_5}{-\sqrt{ab}} \lambda^{-1/4}
    \end{array}
  \right).
  $$ 
     Let $Z=A^{-1}B$ denote the  normalized period  matrix of $X_{\lambda}$. Then, it follows that 
     
      \begin{align*}
Z &\sim \frac{  {ab}\lambda^{1/4}} { { 2c_3}  c_5} \left(
    \begin{array}{cc}
   \frac{-2c_3}{\sqrt{ab}}  & -c_6 \\
       \frac{c_4}{\sqrt{ab}} \lambda^{1/4} & \frac{c_5}{-\sqrt{ab}} \lambda^{-1/4}
    \end{array}
  \right) \frac{1}{\sqrt{-ab}}\left(
    \begin{array}{cc}
   c_5 \lambda^{-1/4}   & d_3\sqrt{-ab}\\
    -c_4 \lambda^{1/4}  & d_4 \sqrt{-ab}
    \end{array}
  \right)\\
&\sim \left(
    \begin{array}{cc}
     \sqrt{-1}   &    \sqrt{-1} \lambda^{1/4}  { \left(   d_3\sqrt{-ab} - \frac{c_6 d_4  ab}{2c_3}  \right) } {  c_5^{-1}}   \\
    \sqrt{-1} \lambda^{1/4} \frac{ c_4 } { {- c_3}   }  &   \tau
    \end{array}
  \right),\\
   \Imagine Z\sim&\left(
    \begin{array}{cc}
1 &    \Hi( \lambda^{1/4})\\
  \Real \left\{   \lambda^{1/4} \frac{ c_4 } {  {- c_3}  }  \right\}& \Imagine \tau
    \end{array}
  \right). 
\end{align*}
So, as $\lambda\rightarrow 0$,
$$ (\Imagine Z)^{-1}\sim    \left(
    \begin{array}{cc}
1   &  -({\Imagine \tau })^{-1}  \Hi( \lambda^{ {1}/{4}}) \\
    ({\Imagine \tau })^{-1}  \Real \left\{   \lambda^{1/4} \frac{  c_4 } {   c_3}  \right\}  &  ({\Imagine \tau })^{-1}
    \end{array}
  \right).$$ 
  By \eqref{Bergman-2-cuspII} and \eqref{reg-2-cusp},
  \begin{equation}
k_{\lambda}(z) \sim  4  \frac{1  + ({\Imagine \tau })^{-1}  \Real \left\{   \lambda^{1/4} \frac{  c_4 } {  c_3} \right\} (\overline z^{2 } + {z}^{2}) + (\Imagine \tau)^{-1}  |z|^{4}  }{ |(z^2-a)(z^2-b) (z^2-\lambda)(z^2-\lambda^2)|} \to    \left(\frac{\Imagine \tau     }{ |z^4  |}   +1\right) k_0 (z),
  \end{equation}
  which means that 
$\kappa_{X_\lambda}  \not \to \kappa_0.
$ Moreover, in the local coordinate $z=\sqrt x$ near $(0, 0)$,
$$
{k}_{\lambda}(z) - \left(\frac{\Imagine \tau     }{ |z^4  |}   +1\right) k_0 (z) \sim \frac{ 4  ({\Imagine \tau })^{-1}  \Real \left\{   \lambda^{1/4} \frac{  c_4 } {   c_3}  \right\} (\overline z^{2 } + {z}^{2}) }{ | z^4(z^2-a)(z^2-b) |}   ,
  $$
    which yields the conclusion.

     \end{proof}

\section{Cusp I: hyperelliptic curves}

This section is devoted to the proof of Theorem \ref{potential-g-cuspI}. When $\lambda\in \mathbb C\setminus \{0, a_1, \ldots, a_{2g-2}\}$, $X_{\lambda}$ defined in \eqref{X-cuspI} has genus $g$. To prove Theorem \ref{potential-g-cuspI}, we use the choice of cycles $\delta_j, \gamma_j$ specified in Section 2, and need to analyze on $X_{\lambda}$ the asymptotics of its $A$-period matrix and $B$-period matrix, denoted by $A$ and $B$, respectively. Meanwhile, for the normalization $Y:=\{(x,y)\in \mathbb C^2\, |\,y^2=xP(x)\} $ of genus $g-1$, denote its $A$-period matrix and $B$-period matrix by $A_0$ and $B_0$, respectively.

\begin{lemma} \label{period-A-hyp-cuspI} Under the same assumptions as in Theorem \ref{potential-g-cuspI}, as $\lambda\rightarrow 0$, it holds that 
 $$
A \sim \left(
    \begin{array}{cccccc}
c_5 \lambda^{-\frac{1}{4}} \frac{1}{\sqrt{P(0)}}  & \Hi(1) \\
\diamond &      A_0\\
    \end{array}
  \right),
$$
where $\diamond$ is a column vector with $g-1$ rows whose entries are given by \eqref{*} below.
 \end{lemma}
 
  \begin{proof} We will estimate all the $g\times g$ elements one by one. Firstly, as $\lambda\to 0$, it holds that \begin{align*}
A_{11}&=-2\int_{-\sqrt{\lambda}}^0 \frac{dx}{\sqrt{x(x^2-{\lambda})P(x)}} \\
&=-2\int_{-\sqrt{\lambda}}^0 \frac{dx}{\sqrt{x(x^2-{\lambda})}}\frac{1}{\sqrt{P(0)}}(1+\Hi(x))\\
& =  -2\lambda^{-\frac{1}{4}} \int_{0}^{1} \frac{  du}{\sqrt{ u(  u-1)(  u-2 )}} \frac{(1+\Hi((u-1)\sqrt{\lambda}))}{\sqrt{P(0)}}\\
&\sim-2 \lambda^{-\frac{1}{4}}\int_{0}^{1} \frac{  du}{\sqrt{ u(  u-1)(  u-2 )}} \frac{1}{\sqrt{P(0)}} =: \lambda^{-\frac{1}{4}} \frac{c_5}{\sqrt{P(0)}},
\end{align*}
where the third equality holds due to the 
substitution
$
x= (u-1) \sqrt{\lambda}.
$
In general, for $a_{i1}$, $2\leq i \leq g$, there is an extra $x^{i-1}$ in the original integrand and thus an extra $\sqrt{\lambda}^{i-1} (u-1)^{i-1}$ in the numerator of the above last expression, so
\begin{equation} \label{*}
A_{i1}\sim -2\int_{0}^{1} \frac{ (u-1)^{i-1} du}{\sqrt{ u(u-1)(u-2)}} \frac{\sqrt{\lambda}^{\frac{1}{2}+i-2}}{\sqrt{P(0)}}.
\end{equation}
Secondly, let $\delta_2$ contain only $-\sqrt{\lambda}, 0, \sqrt{\lambda}$ and $a_1$ (but $a_2,\ldots, a_{2g-2}$). Then,
\begin{align*}
A_{12}&= \int_{\delta_2} \frac{dx}{\sqrt{x(x^2-\lambda)P(x)}}\\
&=\int_{\delta_2} \frac{dx}{x \sqrt{x P(x) }} \left( 1 +\Hi \left(\frac{{\lambda} }{{x}^2}\right)\right) \\
& \sim \int_{\delta_2} \frac{dx}{x \sqrt{x P(x) }}.
\end{align*}  
  For general $a_{i2} $, there is an extra $x^{i-1}$ in the original integrand, so
$$
A_{i2} \sim \int_{\delta_2} \frac{x^{i-2} dx}{  \sqrt{x P(x) }}.
$$
 In general, for $j\geq 2$, let $\delta_j$ contain only $-\sqrt{\lambda}, 0, \sqrt{\lambda}, a_1, \ldots, a_{2j-3}$ (but $a_{2j-2}, \ldots, a_{2g-2}$). 
The entry $A_{ij}$ is asymptotic to the same integrand along $\delta_j$ instead of along $\delta_2$, i.e., 
$$
A_{ij} \sim \int_{\delta_j} \frac{x^{i-2} dx}{  \sqrt{x P(x) }},
$$
which is exactly the same as the corresponding entry of $A_0$ when $i\geq 2$.

\end{proof}

\begin{lemma}  \label{period-B-hyp-cuspI}  Under the same assumptions as in Theorem \ref{potential-g-cuspI}, as $\lambda\rightarrow 0$,
it holds that
$$ B \sim \left(
    \begin{array}{cccccc}
c_5 \lambda^{-\frac{1}{4}}  \frac{\sqrt{-1}   }{ \sqrt{P(0)}} & \Hi(1) \\
    \diamondsuit & B_0 
        \end{array}
  \right),
$$
where $\diamondsuit$ is a column vector whose entries are $  B_{i1} = (-1)^{i-1} \sqrt{-1} A_{i1}$ for $A_{i1}$ in \eqref{*} and $2\leq i\leq g$. 
\end{lemma}

  \begin{proof} For the first column of $B$, we make the change of coordinates (similar to the proof of Lemma \ref{period-B-2-cuspI}) by setting $x=(-u+1) \sqrt{\lambda}$, and get for $1\leq i \leq g$ that 
  $$
  B_{i1}\sim (-1)^{i-1} \sqrt{-1} a_{i1}.
  $$ 
For Column $j$, $2 \leq j \leq g$, we use Taylor expansion of $(x^2-\lambda)^{-1/2}$ and get that
 $$
B_{ij} \sim -2\int_{a_{2j-3}}^{a_{2j-2}} \frac{x^{i-2} dx}{  \sqrt{x P(x) }}.
$$

\end{proof}

   Now we will give a proof of Theorem \ref{potential-g-cuspI}.

\begin{proof} [\bf Proof of Theorem \ref{potential-g-cuspI}] 

  By Lemma \ref{period-A-hyp-cuspI} and the block matrix inversion, we know that

$$ 
A^{-1} \sim \left(
    \begin{array}{cccccc}
    \lambda^{\frac{1}{4}} \frac{\sqrt{P(0)}}{c_5 }  & \Hi( \lambda^{ {1}/{4}})  \\
  ( -A_0^{-1} \diamond) \lambda^{\frac{1}{4}} \frac{\sqrt{P(0)}}{c_5 }      & A_0^{-1}
    \end{array}
  \right),
$$
 where $\lim\limits_{\lambda \to 0}  \frac{(\Hi(\lambda^{\frac{1}{2}}))^T}{\lambda^{\frac{1}{2}}}$ is a finite column vectors with $g-1$ rows.
Therefore,  as $\lambda\to 0$,
$$
Z=A^{-1}B \sim  \left(
    \begin{array}{cccc}
   \sqrt{-1}   & \Hi( \lambda^{\frac{1}{4}})\\
      A_0^{-1} (\diamondsuit -\sqrt{-1}\diamond) & A_0^{-1} B_0
    \end{array}
  \right), 
  \quad 
  \Imagine Z \sim  \left(
    \begin{array}{cccc}
1 & \Hi( \lambda^{ {1}/{4}}) \\
     \Imagine (A_0^{-1} (\diamondsuit -\sqrt{-1}\diamond)) &  \Imagine  Z_0
     \end{array}
  \right),
  $$
     and
  $$ (\Imagine Z)^{-1}\sim  \left(
    \begin{array}{cccc}
1 & \Hi ( \lambda^{ {1}/{4}})     \\
   (\Imagine  Z_0 )^{-1}   \Imagine (A_0^{-1} (\diamondsuit -\sqrt{-1}\diamond))  &  (\Imagine  Z_0 )^{-1}   
    \end{array}
  \right).
  $$
By \eqref{Bergman-g-cuspI} and \eqref{Bergman-reg-cusp}, as $\lambda \to 0$, in the local coordinate $z=\sqrt x$ near $(0, 0)$, it holds that
  $$
 {k}_{\lambda}(z)  \to \frac{ 1  +\sum _{i,j=1}^{g-1}  ((\Imagine  Z_0 )^{-1})_{i, j}   ( z^{i} \overline{z}^{j})^2  }{ |  4^{-1}z^4 P(z^2)|}     =   {k}_{0}(z) + \frac{4}{|  z^4 P(z^2)|},
 $$
so
$\kappa_{X_\lambda}  \not \to \kappa_0$. Moreover, 
  $$
 {k}_{\lambda}(z)  - {k}_{0}(z)  -\frac{4}{|  z^4 P(z^2)|}  \sim \frac{ - 8 \Real  \sum\limits_{i=1}^{g-1}   \left((\Imagine  Z_0 )^{-1}   \Imagine (A_0^{-1} (\diamondsuit -\sqrt{-1}\diamond))  \right)_{i} {z}^{2i}  }{ |  z^4  P(z^2)|},$$
 which yields the conclusion.

 \end{proof}
  
Since $Z$ is symmetric, $A_0^{-1} (\diamondsuit -\sqrt{-1}\diamond)=\Hi ( \lambda^{ {1}/{4}})$, and both the leading and subleading terms in the expansion of $\kappa_{X_\lambda}$ are harmonic with respect to $\lambda$.

\section{Cusp II: genus-two curves}
In this section, we consider a family of genus two curves 
\begin{equation} \label{X-cuspII-2}
X_{\lambda}:=\{(x,y)\in \mathbb C^2\, |\,y^2=x(x-{\lambda}) (x-1)(x-a)(x-b)\}, 
\end{equation}
where $a, b, \lambda$ are distinct complex numbers satisfying $0<|\lambda| <|a|<|b|$. As $\lambda \to 0$, $X_{\lambda}$ degenerates to a singular curve $X_0$ with an ordinary  cusp. The normalization of $X_0$ is an elliptic curve $\{(x,y)\in \mathbb C^2\, |\,y^2=x(x-a)(x-b) \} $, whose period is $\tau$ given in \eqref{period-tau}. The Bergman kernels on $X_{\lambda}$ and on the normalization of $X_0$ are denoted by $\kappa_ {X_\lambda}$ and $\kappa_0$, respectively. In the local coordinate $z:=\sqrt{x}$ near $(0, 0)$, we write $\kappa_ {X_\lambda}=k_{\lambda}(z) dz\otimes d\bar z$, and $\kappa_ 0=k_0(z) dz\otimes d\bar z$. Then, our result on the asymptotic behaviour of the Bergman kernel $\kappa_{X_\lambda}$ with precise coefficients is stated as follows.

 \begin{theorem} \label{potential-2-cuspII}  
For $X_{\lambda}$ defined by \eqref{X-cuspII-2}, as $\lambda\rightarrow 0$, it holds that
\item [(i)] $ \kappa_ {X_\lambda} \to \kappa_0$;

\medskip

\item [(ii)] for small $|z| \neq 0$ , 
 $$
 \log {k} _{\lambda}(z) - \log {k}_0(z) \sim \frac{\pi \cdot \Imagine  \tau    }{-\log|\lambda|\cdot |z|^4},
 $$
  \end{theorem}

  To prove Theorem \ref{potential-2-cuspII}, we need the following two lemmata by analyzing the asymptotics of the $A$-period matrix and $B$-period matrix on $X_\lambda$. Our choice of cycles $\delta_j, \gamma_j$ is specified in Section 2.

\begin{lemma} \label{period-A-2-cuspII} Under the same assumptions as in Theorem \ref{potential-2-cuspII}, as $\lambda\rightarrow 0$, it holds that
  $$ A \sim \left(
    \begin{array}{cc}
      \frac{-2 \pi}{\sqrt{ab}\sqrt{\lambda}} & c_6 \\
      c_7 \lambda^{3/2} & \frac{-2c_3}{\sqrt{ab}}
    \end{array}
  \right),$$ where $c_7:=\frac{-2}{\sqrt{ab}}\int_{0}^{1} \sqrt{ v^{-1}-1}dv$ and $c_3, c_6$ are the same as in Lemma \ref{period-A-2-cuspI}.
 \end{lemma}

\begin{proof}  

Letting ${x=\lambda^2(1-v)}$, 
we consider the integrals 
$$
 \int_0^{\lambda^2} \frac{dx}{\sqrt{x(x-\lambda)(x-\lambda^2)}}  =   \int_{0}^{1} \frac{ -\sqrt{-1} \lambda^{-1} dv}{ \sqrt{v(v-1) (v-1+\lambda^{-1})}} \sim  \frac{-\lambda^{-1}}{2  \sqrt{-1}}\int_{\mathcal C } \frac{ dv}{ v \lambda^{-1/2}}=  \frac{- \pi}{ \sqrt{\lambda}},
$$
where $\mathcal C$ is a large cycle containing $0, 1$, and
\begin{align*}
 \int_0^{\lambda^2} \frac{xdx}{\sqrt{x(x-\lambda)(x-\lambda^2)}} &
=\frac{-\lambda} { \sqrt{-1}}\int_{0}^{1} \frac{ (v-1)dv}{ \sqrt{v(v-1) (v-1+  \lambda^{-1})}}\sim  \frac{-\lambda^{3/2} } { \sqrt{-1}}\int_{0}^{1} \sqrt{\frac{ v-1}{v}}dv.
\end{align*}
We estimate the four entries one by one. Firstly, let $\delta_1$ only contain $0$ and $\lambda^2$, and similar to the proof of Lemma \ref{period-A-2-cuspI} we know that
\begin{align*} 
A_{11}&=
-2\int_0^{\lambda^2}\frac{dx}{\sqrt{x(x-\lambda)(x-\lambda^2)}}\frac{1}{\sqrt{-a}}\frac{1}{\sqrt{-b}}\left(1+\frac{x}{2a}+\Hi(x^2)\right)\left(1-\frac{x}{2b}+\Hi(x^2)\right)\\
&\sim\frac{-2}{-\sqrt{ab}}\int_0^{\lambda^2} \frac{dx}{\sqrt{x(x-\lambda)(x-\lambda^2)}} \\
&\sim
\frac{-2 \pi}{\sqrt{ab}\sqrt{\lambda}}.  
\end{align*}
Secondly,  
$$
A_{21} \sim  \frac{2}{\sqrt{ab}}\int_0^{\lambda^2} \frac{xdx}{\sqrt{x(x-\lambda)(x-\lambda^2)}} \sim \frac{-2}{\sqrt{ab}}\frac{\lambda^{3/2} } { \sqrt{-1}}\int_{0}^{1} \sqrt{\frac{ v-1}{v}}dv=:c_7 \lambda^{3/2}.
$$
Thirdly, let $\delta_2$ contain only $0, \lambda, \lambda^2$ and $a$. Then, it holds that \begin{align*}
A_{12}=& \int_{\delta_2} \frac{dx}{\sqrt{x(x- \lambda) (x-\lambda^2)(x-a)(x-b)}}\\
=& \int_{\delta_2} \frac{dx}{x\sqrt{x(x-a)(x-b)}} \left( 1+\Hi\left( \frac{\lambda}{x}\right)\right)  \left( 1+\Hi\left( \frac{\lambda^2}{x}\right)\right)\\
\sim &\int_{\delta_2} \frac{dx}{x\sqrt{x(x-a)(x-b)}}=: c_6,
\end{align*}
Lastly, 
$$
A_{22} \sim \int_{\delta_2} \frac{dx}{x\sqrt{x(x-a)(x-b)}}=\int_0^a \frac{dx}{x\sqrt{x(x-a)(x-b)}}.
$$

\end{proof}

\begin{lemma} \label{period-B-2-cuspII} Under the same assumptions as in Theorem \ref{potential-2-cuspII}, as $\lambda\rightarrow 0$, it holds that $$
  B \sim \left(
    \begin{array}{cc}
     \frac{2 \sqrt{-1} \log \lambda}{\sqrt{ab} \sqrt{\lambda}}   &  d_3\\
    \frac{2}{\sqrt{ab}}\sqrt{-\lambda}& d_4
    \end{array}
  \right),$$ where $d_3, d_4$ are the same as in Lemma \ref{period-B-2-cuspI}. 
  
    \end{lemma}
 
\begin{proof}  
By \cite{D4} or \eqref{I}, it is known that, as $\lambda \to 0$,
$$
\int_{\lambda^2}^{ \lambda}\frac{dx}{\sqrt{x(x-\lambda)(x-\lambda^2)}} =\frac{1}{ \sqrt{\lambda}} \int_\lambda^{1} \frac{du}{\sqrt{u(u-1)(u-\lambda)}}\sim \frac{ \sqrt{-1} \log \lambda}{ \sqrt{\lambda}}.
$$
Firstly, let $\gamma_1$ contain only $\lambda$ and $\lambda^2$, and we get that 
\begin{align*} 
B_{11}& =-2\int_{\lambda^2}^{\lambda} \frac{dx}{\sqrt{x(x-\lambda)(x-\lambda^2)(x-a)(x-b)}}\\
&=\frac{-2}{-\sqrt{ab}}\int_{\lambda^2}^{\lambda}\frac{dx}{\sqrt{x(x-\lambda)(x-\lambda^2)}}\left(1+\Hi(x)\right)\\
&\sim \frac{2}{\sqrt{ab}}\int_{\lambda^2}^{\lambda}\frac{dx}{\sqrt{x(x-\lambda)(x-\lambda^2)}} \\
& \sim   \frac{2}{\sqrt{ab}} \frac{ \sqrt{-1} \log \lambda}{ \sqrt{\lambda}}.
\end{align*}
Secondly, by \eqref{II} and the substitutions $t= {\lambda^{-1}}$ and $x= \lambda {s^{-1}}$, as $\lambda \to 0$,
$$
B_{21}\sim\frac{2}{\sqrt{ab}}\int_{\lambda^2}^{\lambda}\frac{xdx}{\sqrt{x(x-\lambda)(x-\lambda^2)}} = \frac{-2}{\sqrt{ab}} \int_1^{t}\frac{ ds }{s^2 \sqrt{s-t}}
 \sim \frac{2}{\sqrt{ab}}\sqrt{-\lambda}.
 $$
Thirdly, 
\begin{align*}
B_{12}=& -2\int_a^b \frac{dx}{\sqrt{x(x-\lambda)(x-\lambda^2)(x-a)(x-b)}}\\
= &-2 \int_a^b  \frac{dx \left( 1+\Hi\left({\lambda}{x^{-1}}\right)\right)\left( 1+\Hi\left( {\lambda^2}{x^{-1}}\right)\right)}{x\sqrt{x(x-a)(x-b)}} \\
\sim &  -2 \int_a^b \frac{dx}{x\sqrt{x(x-a)(x-b)}}:=d_3.
\end{align*}  
Lastly,
$$
B_{22} \sim  -2 \int_a^b \frac{dx}{\sqrt{x(x-a)(x-b)}}:=d_4.
$$
\end{proof}

\begin{proof} [\bf Proof of Theorem \ref{potential-2-cuspII}] 
On the regular part of the cuspidal curve $X_0$, \eqref{reg-2-cusp} gives the formula for the Bergman kernel $\kappa _0 =k_0(z)|dz|^2$ in the local coordinate $z=\sqrt x$. By Lemma \ref{period-A-2-cuspII}, as $\lambda\to 0$,
  $$
  A^{-1}\sim   \left(
    \begin{array}{cc}
      \frac{-2 \pi}{\sqrt{ab}\sqrt{\lambda}} & c_6 \\
      c_7 \lambda^{3/2} & \frac{-2c_3}{\sqrt{ab}}
    \end{array}
  \right)^{-1} \sim \frac{ab\sqrt{\lambda}}{ { 4 \pi}{ }  {  c_3}  } \left(
    \begin{array}{cc}
\frac{-2c_3}{\sqrt{ab}} & -c_6   \\
     -  c_7 \lambda^{3/2} & \frac{-2 \pi}{\sqrt{ab}\sqrt{\lambda}}
    \end{array}
  \right).
  $$ 
   Let $Z=A^{-1}B$ denote the normalized period matrix of $X_{\lambda}$. Then, it follows that $$
Z 
\sim  
   \left(
    \begin{array}{cc}
        \frac{ -\sqrt{-1} \log \lambda}{\pi }  & \frac{{ {ab}\sqrt{\lambda}}}{  4 \pi   c_3   }  \left(  \frac{-2c_3}{\sqrt{ab}}d_3- c_6  d_4 \right) \\
   \frac{\sqrt{\lambda}}{  \sqrt{-1} } c_3 ^{-1}     { }{   }    & \tau   
    \end{array}
  \right), \quad  \Imagine Z\sim \left(
    \begin{array}{cc}
      \frac{-\log |\lambda|}{\pi }  & \Hi( \lambda^{\frac{1}{2}})\\
   - \Real \left\{  \sqrt{\lambda} c_3^{-1}  \right\} & \Imagine \tau
    \end{array}
  \right) 
  $$ 
So, as $\lambda\rightarrow 0$,
$$ (\Imagine Z)^{-1}\sim    \left(
    \begin{array}{cc}
     \frac{\pi}{- \log |\lambda|  }     & \frac{1}{  \log |\lambda|  } \Hi( \lambda^{\frac{1}{2}}) \\
    \frac{-\pi } { \log |\lambda| \Imagine \tau}     \Real \left\{  \sqrt{\lambda} c_3^{-1}  \right\}   &  ({\Imagine \tau })^{-1}
    \end{array}
  \right).$$ 
By \eqref{Bergman-2-cuspII} and \eqref{reg-2-cusp},
$\kappa _{X_\lambda} \to \kappa _0.
$ Moreover, in the local coordinate $z=\sqrt x$ near $(0, 0)$,
$$
{k}_{\lambda}(z) -  {k}_{0}(z) \sim \frac{4 \pi}{ | (z^2-a)(z^2-b) (z^2-\lambda)(z^2-\lambda^2)|}  \cdot  \frac{  1+ \Hi( \lambda^{\frac{1}{2}})   }{ - \log |\lambda|},
  $$
    which yields the conclusion.

     \end{proof}

\section{Cusp II: hyperelliptic curves}

This section is devoted to the proof of Theorem \ref{potential-g-cuspII}. When $\lambda\in \mathbb C\setminus \{0, 1, a_1, \ldots, a_{2g-2}\}$, $X_{\lambda}$ defined in \eqref{X-cuspII} has genus $g $. To prove Theorem \ref{potential-g-cuspII}, we use the choice of cycles $\delta_j, \gamma_j$ specified in Section 2, and need to analyze on $X_{\lambda}$ the asymptotics of its $A$-period matrix and $B$-period matrix, denoted by $A$ and $B$, respectively. Meanwhile, for the normalization $Y:=\{(x,y)\in \mathbb C^2\, |\,y^2=xP(x)\} $ of genus $g-1$, denote its $A$-period matrix and $B$-period matrix by $A_0$ and $B_0$, respectively.

\begin{lemma} \label{period-A-hyp-cuspII} Under the same assumptions as in Theorem \ref{potential-g-cuspII}, as $\lambda\rightarrow 0$, it holds that
 $$
A \sim \left(
    \begin{array}{cccccc}
     \frac{2\pi}{\sqrt{\lambda P(0) }} & \Hi(1) \\
** &      A_0\\
    \end{array}
  \right),
$$
where $**$ is a column vector with $g-1$ rows whose entries are at most $\Hi(\lambda^{3/2}).$

\end{lemma}
 
  \begin{proof} 
  Firstly, as $\lambda\rightarrow 0$,
  \begin{align*}
A_{11}&
=-2\int_0^{ \lambda^2} \frac{dx}{\sqrt{x(x-{\lambda})(x-{\lambda}^2)}}\frac{1}{\sqrt{P(0)}}(1+\Hi(x))\\
&\sim -2\int_0^{ \lambda^2} \frac{dx}{\sqrt{x(x-{\lambda})(x-{\lambda}^2)}}\frac{1}{\sqrt{P(0)}} \\
& \sim \frac{2\pi}{\sqrt{\lambda P(0) }}.
\end{align*}
$$
A_{21}\sim \frac{-2}{\sqrt{P(0)}}  \int_0^{ \lambda^2} \frac{xdx}{\sqrt{x(x-{\lambda})(x-{\lambda}^2)}} \sim  \frac{-2}{\sqrt{P(0)}} \frac{-\lambda^{3/2} } { \sqrt{-1}}\int_{0}^{1} \sqrt{\frac{ v-1}{v}}dv.
$$
Here we change the variable by letting $x=\lambda^2(1-v)$. In general, for $2\leq i\leq g$,  
$$
A_{i1}\sim  \frac{-2}{\sqrt{P(0)}} \frac{-\lambda^{3/2} } { \sqrt{-1}}\int_{0}^{1} {\lambda}^{2i-4} (1-v)^{i-2} \sqrt{\frac{ v-1}{v}}dv=\Hi(\lambda^{2i-2.5}).
$$
Secondly, let $\delta_2$ contain only $0, \lambda^2, \lambda$ and $a_1$ (but $a_2,\ldots, a_{2g-2}$). Then,
\begin{align*}
A_{12}&= \int_{\delta_2} \frac{dx}{\sqrt{x(x-{\lambda})(x-{\lambda}^2)P(x)}}\\
&=\int_{\delta_2} \frac{dx}{x \sqrt{x P(x) }} \left( 1 +\Hi \left(\frac{{\lambda} }{{x}^2}\right)\right) \\
& \sim \int_{\delta_2} \frac{dx}{x \sqrt{x P(x) }}.
\end{align*}  
In general, for $A_{i2} $, there is an extra $x^{i-1}$ in the original integrand above, so
$$
A_{i2} \sim \int_{\delta_2} \frac{x^{i-2} dx}{  \sqrt{x P(x) }}.
$$
Thirdly, for $j\geq 2$, let $\delta_j$ contain only $0, \lambda^2, \lambda, a_1, \ldots, a_{2j-3}$ (but $a_{2j-2}, \ldots, a_{2g-2}$).  In general, the entry $A_{ij}$ is asymptotic to the same integrand along $\delta_j$ instead of along $\delta_2$, i.e.,
$$
A_{ij} \sim \int_{\delta_j} \frac{x^{i-2} dx}{  \sqrt{x P(x) }},
$$
which is exactly the same as the corresponding entry of $A_0$ when $i\geq 2$. 

\end{proof}

\begin{lemma} Under the same assumptions as in Theorem \ref{potential-g-cuspII}, as $\lambda\rightarrow 0$, it holds that
$$ B \sim \left(
    \begin{array}{cccccc}
   \frac{-2}{\sqrt{P(0)}} \frac{ \sqrt{-1} \log \lambda}{ \sqrt{\lambda}}  & \Hi(1) \\
    \heartsuit & B_0 
        \end{array}
  \right),
$$
where $ \heartsuit$ is a column vector whose entries are $\frac{-2\lambda^{i-2} \sqrt{-\lambda} }{\sqrt{P(0)}}$, for $2\leq i\leq g$.

\end{lemma}

  \begin{proof} 
    \begin{align*}
B_{11}&=-2\int_{\lambda^2}^\lambda \frac{dx}{\sqrt{x(x-{\lambda})(x-{\lambda}^2)\cdot P(x)}} \\
& =-2\int_{\lambda^2}^\lambda \frac{dx}{\sqrt{x(x-{\lambda})(x-{\lambda}^2)}}\frac{1}{\sqrt{P(0)}}(1+\Hi(x))\\
&\sim -2\int_{\lambda^2}^\lambda \frac{dx}{\sqrt{x(x-{\lambda})(x-{\lambda}^2)}}\frac{1}{\sqrt{P(0)}} \\
& \sim \frac{-2}{\sqrt{P(0)}} \frac{ \sqrt{-1} \log \lambda}{ \sqrt{\lambda}}.
\end{align*}
Thus, for $2\leq i\leq g$, by \eqref{alpha},
\begin{align*}
B_{i1}&=-2\int_{\lambda^2}^\lambda \frac{x^{i-1}dx}{\sqrt{x(x-{\lambda})(x-{\lambda}^2)}}\frac{1}{\sqrt{P(0)}}(1+\Hi(x))\\
&\sim -2\int_{\lambda^2}^\lambda \frac{x^{i-1}dx}{\sqrt{x(x-{\lambda})(x-{\lambda}^2)}}\frac{1}{\sqrt{P(0)}}  \\
& =   \frac{2 \lambda^{i-2} }{\sqrt{P(0)}} \int_1^{t}\frac{ ds }{s^i \sqrt{s-t}} \\
&  \sim \frac{-2\lambda^{i-2} \sqrt{-\lambda} }{\sqrt{P(0)}},
\end{align*}
where the last equality holds due to the substitutions $t= {\lambda^{-1}}$ and $x= \lambda s^{-1}$.
For Column $j$, $2 \leq j \leq g$, we use Taylor expansion of $\sqrt{(x-{\lambda})(x-{\lambda}^2)}^{-1}$ to get that
 $$
B_{ij} \sim -2\int_{a_{2j-3}}^{a_{2j-2}} \frac{x^{i-2} dx}{  \sqrt{x P(x) }}.
$$
which is exactly the same as the corresponding entry of $B_0$ when $i\geq 2$.

\end{proof}

Now we will give a proof of Theorem \ref{potential-g-cuspII}. 

\begin{proof} [\bf Proof of Theorem \ref{potential-g-cuspII}] 

By Lemma \ref{period-A-hyp-cuspII} and the block matrix inversion, we know that
$$ 
A^{-1} \sim \left(
    \begin{array}{cccccc}
   \frac {\sqrt{\lambda P(0) }}{2\pi}  & \Hi( \lambda^{\frac{1}{2}})  \\
     \Hi( \lambda^2) & A_0^{-1}
    \end{array}
  \right),
$$
where both $\lim\limits_{\lambda \to 0}  \frac{(\Hi(\lambda^{\frac{1}{2}}))^T}{\lambda^{\frac{1}{2}}}$ and $\lim\limits_{\lambda \to 0} \frac{\Hi( \lambda^2)}{\lambda^2}$ are finite column vectors with $g-1$ rows.
Therefore,  as $\lambda\to 0$,
$$
Z=A^{-1}B \sim  \left(
    \begin{array}{cccc}
   \frac{-\sqrt{-1} }{\pi}  { \log \lambda}    & \Hi( \lambda^{\frac{1}{2}})\\
      A_0^{-1}\heartsuit& A_0^{-1} B_0
    \end{array}
  \right), 
  \quad 
  \Imagine Z \sim  \left(
    \begin{array}{cccc}
\frac{\log |\lambda|}  {-\pi}  & \Hi( \lambda^{\frac{1}{2}}) \\
     \Imagine (A_0^{-1} \heartsuit) &  \Imagine  Z_0
     \end{array}
  \right),
  $$
  and
  $$ (\Imagine Z)^{-1}\sim  \left(
    \begin{array}{cccc}
 \frac{-\pi}  {\log |\lambda|}    & \Hi (({\log |\lambda|})^{-1}\lambda^{\frac{1}{2}})     \\
   (\Imagine  Z_0 )^{-1}   \Imagine (A_0^{-1} \heartsuit )    \frac{\pi}  {\log |\lambda|}  &  (\Imagine  Z_0 )^{-1}   
    \end{array}
  \right).
  $$
  In fact, since $Z$ is symmetric, the off-diagonal block matrices in each matrix above concerning $Z$ are the transpose of each other.
 On the regular part of the cuspidal curve $X_0$,  the formula for the Bergman kernel $\kappa _0 =k_0(z)|dz|^2$ is given in
 \eqref{Bergman-reg-cusp}. This together with  \eqref{Bergman-g-cuspII}  will imply that $\kappa_ {X_\lambda} \to \kappa_0$, as $\lambda \to 0$. Moreover,  it holds that
$$
{k}_{\lambda}(z) - {k}_{0}(z) \sim
  \frac{4\pi}{|(z^2-\lambda)(z^2-\lambda^2) P(z^2)|}  \frac{ 1- 2 \Real  \sum\limits_{i=1}^{g-1}   \left((\Imagine  Z_0 )^{-1}   \Imagine (A_0^{-1} \heartsuit)  \right)_{i} {z}^{2i}}{ - \log |\lambda|},
  $$
which yields that    
$$
\psi - \log  {k}_0(z)   \sim   \frac{\pi}{-\log |\lambda|}
\frac{  1 }{  \sum\limits_{i,j=1}^{g-1}  ((\Imagine  Z_0 )^{-1})_{i, j}   ( z^{i} \overline{z}^{j})^2}.
$$

 \end{proof}

\section{Jacobian varieties}
We will give a proof of Theorem \ref{Jacobian} by using the results obtained in the proofs of Theorems \ref{potential-g-node}, \ref{potential-g-cuspI}, \ref{potential-g-cuspII}. Let $X_{\lambda}$ be a compact curve of genus $g\geq 2$, and let $Z$ be its period matrix with respect to some chosen homology basis. The Jacobian variety of $X_{\lambda}$, which is denoted by $\text{Jac}(X_{\lambda})$,  is then identified with the $g$-dimensional complex torus $  \mathbb C^g / \mathbb Z^g + Z \mathbb Z^g$. It is well known that the Abel-Jacobi (period) map $X_{\lambda} \to \text{Jac}(X_{\lambda})$ is a holomorphic embedding, and the Bergman kernel on a smooth algebraic curve is the pull back  of the Euclidean metric from the Jacobian variety via this map.

\begin{proof} [\bf Proof of Theorem \ref{Jacobian}] 
By definition \eqref{def}, the Bergman kernel on $\text{Jac}(X_{\lambda})$ can be written as $ \mu_{\lambda}  (dw^1\wedge \cdots \wedge dw^g) \otimes  
(d\overline{w^1}\wedge \cdots \wedge d\overline{w^g}) $, under the coordinate $(w^1, \ldots, w^g)$ induced from $\mathbb C^g$, where $ \mu_{\lambda} =(\det (\Imagine Z) )^{-1}$. After adding a one-point compactification at $\infty$, one may assume that the curve $X_{\lambda}$ is compact.
For $X_{\lambda}$ defined in \eqref{X-node} or \eqref{X-cuspII}, as $\lambda \to 0$, it holds that
$$
\det (\Imagine Z) = 
\frac{\log |\lambda|}  {-\pi}     \det (\Imagine  Z_0) +\Hi(1)\to +\infty,
$$
which yields the first part of the conclusion in Theorem \ref{Jacobian}.  For $X_{\lambda}$ defined in \eqref{X-cuspI}, 
more careful analysis in the proof of Theorem \ref{potential-g-cuspI} shows that as $\lambda \to 0$,
 $$
  \Imagine Z  =  \left(
    \begin{array}{cccc}
1+ \Hi( \lambda^{ {1}/{2}}) & \Hi( \lambda^{ {1}/{4}}) \\
    \Hi( \lambda^{ {1}/{4}}) &  \Imagine  Z_0 + \Hi( \lambda^{ {1}/{2}})
     \end{array}
  \right), \quad \det (\Imagine Z)    = \det \Imagine  Z_0  + \Hi( \lambda^{ {1}/{2}}) < +\infty,
 $$
which yields the second part of the conclusion.

  \end{proof}

\subsection*{Funding}

{\fontsize{11.5}{10}\selectfont

The research of the author is supported by AMS-Simons travel grant. This work was supported by KAKENHI and the Grant-in-Aid for JSPS Fellows (No. 15J05093).
}

\subsection*{Statements and Declarations}

No financial or non-financial interests that are directly or indirectly related to the work submitted for publication was reported by the author. 

Data sharing not applicable to this article as no datasets were generated or analysed during the current study.

\subsection*{Acknowledgements} 
{\fontsize{11.5}{10}\selectfont
This article constitutes essentially the author's doctoral thesis in Nagoya University under the advice of Professor Ohsawa. The author sincerely thanks the  referee for many helpful suggestions and   comments.}

\fontsize{11}{9}\selectfont

\vspace{0.5cm}

\noindent dong@uconn.edu, 

\vspace{0.2 cm}

\noindent Department of Mathematics, University of Connecticut, Storrs, CT 06269-1009, USA


\begin{thebibliography}{99}

 \fontsize{10pt}{9pt} \selectfont
 
\setlength{\parskip}{1ex}
\setlength{\itemsep}{1.2ex}      

\bibitem {B1} B. Berndtsson, \emph{Subharmonicity properties of the Bergman kernel and some other functions associated to pseudoconvex domains}, Ann. Inst. Fourier (Grenoble) {\bf 56} (2006), 1633--1662.

\bibitem {B2} B. Berndtsson, \emph{Curvature of vector bundles associated to holomorphic fibrations}, Ann. of Math. {\bf 169} (2009), 531--560.

\bibitem {B3} B. Berndtsson, \emph{Positivity of direct image bundles and convexity on the space of K\"{a}hler metrics}, J. Differential Geom. {\bf 81} (2009), 457--482. 

\bibitem {B4} B. Berndtsson, \emph{Strict and nonstrict positivity of direct image bundles}, Math. Z.  {\bf 269} (2011), 1201--1218.

\bibitem{B5} B. Berndtsson, \emph{The openness conjecture and complex Brunn-Minkowski inequalities}, {Complex Geometry and Dynamics}, 29-44, Abel Symp., 10, Springer, Cham, 2015.

\bibitem {BL} B. Berndtsson and L. Lempert, \emph{A proof of the Ohsawa-Takegoshi theorem with sharp estimates}, J. Math. Soc. Japan {\bf 68} (2016), 1461--1472.
 
\bibitem {BP} B. Berndtsson and M. P\v{a}un, \emph{Bergman kernel and the pseudoeffectivity of relative canonical bundles}, Duke Math. J. {\bf 145} (2008), 341--378.

\bibitem {BP2} B. Berndtsson and M. P\v{a}un, \emph{Quantitative extensions of pluricanonical forms and closed positive currents}, Nagoya Math. J. {\bf 205} (2012), 25--65. 

\bibitem {Bl13} Z. B\l{}ocki, \emph{Suita conjecture and the Ohsawa-Takegoshi extension theorem}, Invent. Math. {\bf193} (2013), 149--158.

\bibitem {BJ17} {S. Boucksom and M. Jonsson}, \emph{Tropical and non-Archimedean limits of degenerating families of volume forms}, J. \'Ec. polytech. Math. {\bf 4} (2017), 87--139.

\bibitem {Ca} J. Cao, \emph{Ohsawa-Takegoshi extension theorem for compact K\"{a}hler manifolds and applications}, Complex and Symplectic Geometry, 19--38, Springer INdAM Ser., 21, Springer, Cham, 2017.

\bibitem {CMSP} J. Carlson, S. M\"{u}ller-Stach and C. Peters, \emph{Period Mappings and Period Domains}, 2nd edition, Cambridge Stud. Adv. Math. {\bf 168}, Cambridge University Press, Cambridge, 2017.

\bibitem {De} P. Deligne, \emph{Th\'eorie de Hodge, II} (French), Publ. Math. Inst. Hautes \'Etudes Sci. {\bf 40} (1971), 5--57.
 
\bibitem {Dem12} J.-P. Demailly, \emph{Analytic Methods in Algebraic Geometry}, Surv. Mod. Math., 1, International Press, Somerville, MA; Higher Education Press, Beijing, 2012.

\bibitem {Dem16} {J.-P. Demailly}, \emph{Extension of holomorphic functions defined on non reduced analytic subvarieties}, The legacy of Bernhard Riemann after one hundred and fifty years, I, 191--222, Adv. Lect. Math. (ALM), 35.1, Int. Press, Somerville, MA, 2016.

\bibitem {dJ} R. de Jong, \emph{Faltings delta-invariant and semistable degeneration}, J. Differential Geom. {\bf 111} (2019), 241--301.

\bibitem {DWZZ} {F. Deng, Z. Wang, L. Zhang and X. Zhou}, \emph{New characterizations of plurisubharmonic functions and positivity of direct image sheaves}, {\color{blue} arXiv: 1809.10371}.

\bibitem{DS} S. Donaldson and S. Sun, \emph{Gromov-Hausdorff limits of K\"{a}hler manifolds and algebraic geometry}, Acta Math. {\bf 213} (2014), 63--106.

\bibitem{DS2} S. Donaldson and S. Sun, \emph{Gromov-Hausdorff limits of K\"{a}hler manifolds and algebraic geometry, II}, J. Differential Geom. {\bf 107} (2017), 327--371.
 
\bibitem {D1} R. X. Dong, \emph{Boundary asymptotics of the relative Bergman kernel metric for elliptic curves}, C. R. Math. Acad. Sci. Paris  {\bf 353} (2015), 611--615.

\bibitem {D2} R. X. Dong, \emph{Boundary asymptotics of the relative Bergman kernel metric for elliptic curves II: subleading terms}, Ann. Polon. Math. {\bf 118} (2016), 59--69.

\bibitem {D3} R. X. Dong, \emph{Boundary asymptotics of the relative Bergman kernel metric for elliptic curves III: $1\, \&\, \infty$}, J. Class. Anal. {\bf 9} (2016), 61--67.

\bibitem {D4} R. X. Dong, \emph{Boundary asymptotics of the relative Bergman kernel metric for elliptic curves IV: Taylor series}, Geometric Complex Analysis: In Honor of Kang-Tae Kim’s 60th Birthday, Gyeongju, Korea, 2017, 129--143, Springer Proc. Math. Stat, 246, Springer, Singapore, 2018.
 
 
 
\bibitem {D5} R. X. Dong, \emph{Boundary asymptotics of the relative Bergman kernel metric for hyperelliptic curves}, Complex Manifolds {\bf 4} (2017), 7--15.

\bibitem {D6} {R. X. Dong}, \emph{Equality in Suita's conjecture and metrics of constant Gaussian curvature}, {\color{blue} arXiv:1807.05537}.

\bibitem {DT} {R. X. Dong and J. Treuer}, \emph{Rigidity theorem by the minimal point of the Bergman kernel}, J. Geom. Anal. {\bf 31} (2021), 4856--4864. 

\bibitem {DTZ} {R. X. Dong, J. N. Treuer and Y. Zhang}, \emph{Rigidity theorems by capacities and kernels}, Int. Math. Res. Not. IMRN, to appear. 

 

 



\bibitem {EFM} D. Eriksson, G. Freixas i Montplet and C. Mourougane, \emph{BCOV invariants of Calabi--Yau manifolds and degenerations of Hodge structures}, Duke Math. J. {\bf 170} (2021), 379--454.



\bibitem {F} J. D. Fay, \emph{Theta functions on Riemann surfaces}, Lecture Notes in Math., 352, Springer-Verlag, Berlin-New York, 1973.

\bibitem {Fu} T. Fujita, \emph{On K\"ahler fiber spaces over curves}, J. Math. Soc. Japan {\bf 30} (1978), 779--794.

\bibitem {Gr0} {P. A. Griffiths}, \emph{Periods of integrals on algebraic manifolds: Summary of main results and discussion of open problems}, Bull. Amer. Math. Soc. {\bf 76} (1970), 228--296.

\bibitem {Gr} {P. A. Griffiths}, \emph{Periods of integrals on algebraic manifolds, III (some global differential-geometric properties of the period mapping)}, Publ. Math. Inst. Hautes \'Etudes Sci. {\bf 38} (1970), 125--180.
 


\bibitem {GS} {P. A. Griffiths and W. Schmid}, \emph{Locally homogeneous complex manifolds}, Acta Math. {\bf 123} (1969), 253--302.

\bibitem {GZ1} {Q. Guan and X. Zhou}, \emph{A solution of an $L^2$ extension problem with optimal estimate and applications}, Ann. of Math. {\bf 181} (2015), 1139--1208.

\bibitem {GZ2} {Q. Guan and X. Zhou}, \emph{A proof of Demailly's strong openness conjecture}, Ann. of Math. {\bf 182} (2015), 605--616.

\bibitem {GZ3} {Q. Guan and X. Zhou}, \emph{Effectiveness of Demailly's strong openness conjecture and related problems}, Invent. Math. {\bf 202} (2015), 635--676.

\bibitem {HJ} {L. Habermann and J. Jost}, \emph{Riemannian metrics on Teichm\"uller space}, Manuscripta Math. {\bf 89} (1996), 281--306.

\bibitem {HJ2} {L. Habermann and J. Jost}, \emph{Metrics on Riemann surfaces and the geometry of moduli spaces}, Geometric theory of singular phenomena in partial differential equations (Cortona, 1995), 53--70, Sympos. Math., XXXVIII, Cambridge Univ. Press, Cambridge, 1998.


\bibitem {J} {J. Jorgenson}, \emph{Asymptotic behavior of Faltings's delta function}, Duke Math. J.  {\bf 61} (1990), 221--254.

 

\bibitem {K} {M. Kashiwara}, \emph{The Asymptotic Behavior of a Variation of Polarized Hodge Structure}, Publ. Res. Inst. Math. Sci. {\bf 21} (1985), 853-875.

\bibitem {KK} {M. Kashiwara and T. Kawai}, \emph{The Poincar\'e lemma for variations of polarized Hodge structure}, Publ. Res. Inst. Math. Sci. {\bf 23} (1987), 345--407.

\bibitem {Kaw} {Y. Kawamata}, \emph{Kodaira dimension of algebraic fiber spaces over curves}, Invent. Math. {\bf 66} (1982), 57--71.

\bibitem {Kim} D. Kim, \emph{Canonical bundle formula and degenerating families of volume forms}, {\color{blue} arXiv:1910.06917}.

\bibitem {Ko} J. Koll\'ar, \emph{Higher direct images of dualizing sheaves I}, Ann. of Math. {\bf 123} (1986), 11--42.


\bibitem {Lel} P. Lelong, \emph{Plurisubharmonic functions and positive differential forms}, Gordon and Breach, New York, and Dunod, Paris, 1969.


\bibitem {Le} M. Levine, \emph{Pluri-canonical divisors on K\"{a}hler manifolds}, Invent. Math. {\bf 74} (1983), 293--304.

\bibitem {Lew} J. Lewittes, \emph{Differentials and Metrics on Riemann Surfaces}, Trans. Amer. Math. Soc.  {\bf 139} (1969), 311–318. 

 

\bibitem {LY} {K. Liu and X. Yang}, \emph{Curvatures of direct image sheaves of vector bundles and applications}, J. Differential Geom. {\bf 98} (2014), 117--145.

\bibitem {MY} F. Maitani and H. Yamaguchi, \emph{Variation of Bergman metrics on Riemann surfaces}, Math. Ann. {\bf 330} (2004), 477--489.

\bibitem {Ma} H. Masur, \emph{The extension of the Weil-Petersson metric to the boundary of Teichmuller space}, Duke Math. J. {\bf 34} (1976), 623--635.
 
  \bibitem {Mum} D. Mumford, \emph{Tata lectures on theta II. Jacobian theta functions and differential equations},  Reprint of the 1984 edition. Mod. Birkh{\"a}user Class, Birkh{\"a}user Boston, Inc., Boston, MA, 2007. 


 
 \bibitem {Nag} S. Nag, \emph{Riemann surfaces and their Jacobians: a toolkit}, Indian J. Pure Appl. Math. {\bf 24} (1993), 729--745. 



\bibitem {O95} {T. Ohsawa}, \emph{Addendum to ``On the Bergman kernel of hyperconvex domains"}, Nagoya Math. J. {\bf 137} (1995), 145--148.

\bibitem {O17} {T. Ohsawa}, \emph{An update of extension theorems by the $L^2$ estimates for $\bar\partial$}, Hodge theory and $L^2$-analysis, 489--516, Adv. Lect. Math., 39, Int. Press, Somerville, MA, 2017.

\bibitem {O18} T. Ohsawa, \emph{$L^2$ approaches in several complex variables. 
Towards the Oka-Cartan theory with precise bounds.}, 2nd edition, Springer Monogr. Math., Springer, Tokyo, 2018.

\bibitem {O20} {T. Ohsawa}, \emph{A Survey on the $L^2$ Extension Theorems}, J. Geom. Anal. {\bf 30} (2020), 1366--1395.

\bibitem {O20!} {T. Ohsawa}, \emph{A Role of the $L^2$ Method in the Study of Analytic Families}, Bousfield Classes and Ohkawa's Theorem, 423--435, Springer Proc. Math. Stat, 309, Springer, Singapore, 2020.

\bibitem {OT} T. Ohsawa and K. Takegoshi, \emph{On the extension of $L^2$ holomorphic functions}, Math. Z. {\bf 195} (1987), 197--204.

\bibitem {P07} M. P\v{a}un, \emph{Siu's invariance of plurigenera: a one-tower proof}, J. Differential Geom. {\bf 76} (2007), 485--493.

\bibitem {P19} M. P\v{a}un, \emph{Positivit\'e des images directes et applications [d'apr\`es Bo Berndtsson]} (French), S\'eminaire Bourbaki, Vol. 2016/2017, Ast\'erisque {\bf 407} (2019), 53--90.

\bibitem {PT} M. P\v{a}un and S. Takayama, \emph{Positivity of twisted relative pluricanonical bundles and their direct images}, J. Algebraic Geom. {\bf 27} (2018), 211--272.

\bibitem {S} W. Schmid, \emph{Variation of Hodge structure: the singularities of the period mapping}, Invent. Math. {\bf 22} (1973), 211--319.
 
\bibitem {Shi} S. Shivaprasad, \emph{Convergence of Bergman measures towards the Zhang measure},  {\color{blue}  arXiv:2005.05753}.
 
\bibitem {Si} {Y.-T. Siu}, \emph{The Fujita conjecture and the extension theorem of Ohsawa-Takegoshi}, {Geometric Complex Analysis} (Hayama, 1995), 577--592, World Sci. Publ., River Edge, NJ, 1996.

\bibitem {Si98} {Y.-T. Siu}, \emph{Invariance of plurigenera}, Invent. Math. {\bf 134} (1998), 661--673.

\bibitem {Si02} {Y.-T. Siu}, \emph{Extension of twisted pluricanonical sections with plurisubharmonic weight and invariance of semipositively twisted plurigenera for manifolds not necessarily of general type}, Complex geometry (G\"{o}ttingen, 2000), 223--277, Springer, Berlin, 2002.

\bibitem {SSW} {J. Song, J. Sturm and X. Wang}, \emph{Riemannian Geometry of K\"{a}hler-Einstein currents III - Compactness of K\"{a}hler-Einstein manifolds of negative scalar curvature}, {\color{blue}  arXiv:2003.04709}.

\bibitem {Su} N. Suita, \emph{Capacities and kernels on Riemann surfaces}, Arch. Ration. Mech. Anal. {\bf 46} (1972), 212--217.

\bibitem {SZ} S. Sun and R. Zhang, \emph{Complex structure degenerations and collapsing of Calabi-Yau metrics}, {\color{blue}  arXiv:1906.03368}.

\bibitem {Ta} S. Takayama, \emph{Singularities of Narasimhan-Simha type metrics on direct images of relative pluricanonical bundles}, Ann. Inst. Fourier (Grenoble) {\bf 66} (2016), 753--783.

 \bibitem {T} G. Tian, \emph{On Calabi conjecture for complex surfaces with positive first Chern class}, Invent. Math. {\bf 101} (1990), 101--172.

\bibitem {Ts} H. Tsuji, \emph{Curvature semipositivity of relative pluricanonical systems}, {\color{blue} arXiv:0703729}.

 \bibitem {V} C. Voison,  \emph{Hodge theory and complex algebraic geometry. I, II}, Transl. from the French by Leila Schneps, Cambridge Stud. Adv. Math. {\bf 76, 77}, Cambridge Univ. Press, Cambridge, 2007.

\bibitem {We}{R. Wentworth}, \emph{The asymptotic of the Arakelov-Green's function and Faltings' delta invariant}, Comm. Math. Phys. {\bf 137} (1991), 427--459.

\bibitem {Y} A. Yamada, \emph{Precise variational formulas for abelian differentials}, Kodai Math. J.  {\bf 3} (1980), 114--143.

\bibitem {Yo} K.-I. Yoshikawa, \emph{Degenerations of Calabi--Yau threefolds and BCOV invariants}, Internat. J. Math. {\bf 26} (2015), 1540010, 33 pp.

\bibitem {Zu} S. Zucker, \emph{Hodge theory with degenerating coefficients: $L_2$ cohomology in the Poincar\'e metric}. Ann. of Math. {\bf 109} (1979), 415--476.

\end{thebibliography}
\end{document}